\documentclass[11pt,a4paper,reqno]{amsart}
\usepackage[T1]{fontenc}
\usepackage[utf8]{inputenc}

\usepackage{xcolor}
\usepackage[all]{xy}
\usepackage{dsfont}
\usepackage{amsmath,amsfonts,amssymb,amsthm}
\usepackage[bbgreekl]{mathbbol} 
\DeclareSymbolFontAlphabet{\mathbbm}{bbold} 
\DeclareSymbolFontAlphabet{\mathbb}{AMSb} 
\usepackage{mathrsfs}
\usepackage{mathtools}
\usepackage{mathptmx} 
\usepackage{stmaryrd}
\usepackage{enumerate}
\newtheorem{theorem}{Theorem}
\newtheorem{proposition}{Proposition}
\newtheorem{corollary}{Corollary}
\newtheorem{lemma}{Lemma}
\newtheorem{definition}{Definition}
\newtheorem{remark}{Remark}
\newcommand{\R}{\mathds R}
\newcommand{\M}{\mathcal M}
\newcommand{\B}{\mathcal B}
\newcommand{\fiber}{\mathcal F}
\newcommand{\ver}{\mathcal{V}}
\newcommand{\hor}{\mathcal{H}}
\newcommand{\leg}{\mathcal{L}}
\newcommand{\Ad}{\mathcal{A}}
\newcommand{\dif}{\mathrm{d}}
\newcommand{\T}{\mathbf T}
\newcommand{\At}{\mathbf A}
\newcommand{\LB}{\tilde{L}}
\newcommand{\C}{\mathcal C}

\numberwithin{equation}{section}
\usepackage{hyperref}

\usepackage[normalem]{ulem} 

\def\bb#1\eb{\textcolor{blue}{#1}} 
\def\br#1\er{\textcolor{red}{#1}} %
\def\bm#1\em{\textcolor{purple}{#1}} %

\author[M. Huber]{Matthieu Huber}
\address{Departamento de Matem\'aticas, \hfill\break\indent
Universidad de Murcia, Campus de Espinardo,\hfill\break\indent
30100 Espinardo, Murcia, Spain}
\email{matthieu.huber@um.es}

\title{The fundamental equations of a pseudo-Finsler submersion}
\author[M. A. Javaloyes]{Miguel Angel Javaloyes}
\address{Departamento de Matem\'aticas, \hfill\break\indent
Universidad de Murcia, Campus de Espinardo,\hfill\break\indent
30100 Espinardo, Murcia, Spain}
\email{majava@um.es}

\thanks{MAJ was supported by the project  PGC2018-097046-B-I00 funded by MCIN/ AEI /10.13039/501100011033/ FEDER ``Una manera de hacer Europa''.}

\begin{document}



\begin{abstract}
The main result in this paper is the generalisation of the fundamental equations of a Riemannian submersion presented in the 1966 article by Barrett O'Neill \cite{1966ONeill} to the context of pseudo-Finsler submersions.  In the meantime, we also explore some basic properties of the O'Neill fundamental tensors of the submersion and study Finsler submersions with totally geodesic fibers. 
\end{abstract}
\maketitle
\section{Introduction}
Submersions can be thought as a dual concept to submanifolds. The study of the equations describing the geometry of a submanifold of a Riemannian manifold has a long tradition, having the Gauss and Codazzi equations between the most celebrated results in the theory. These equations have been generalized to the Finsler setup from the 1930s  (see \cite{2022HuberJavaloyes} for a summary of the topic). The systematic study of Riemannian submersions is much more recent, having its origin in the foundational paper by O'Neill about the fundamental equations \cite{1966ONeill} (see also \cite{1967Gray}). It is very natural to ask oneself about the generalization to the Finsler realm of these fundamental equations, but the closest achievement until now is the pioneering paper \cite{2001AlDur}, where the authors obtain some  results
using the osculating metric associated with a geodesic vector field. In this paper, we intend to fill this gap providing a complete set of fundamental equations of a pseudo-Finsler submanifold, which describe the curvature tensor of the total space in terms of the configuration O'Neill tensors $\T$ and $\At$ and the curvature tensors of the fibers and the base manifold. We recall that the fundamental equations have been very helpful to compute the curvature tensor in many different situations as for example for models of spacetimes and to find new Einstein manifolds (see \cite[Chapter 9]{Besse}). We hope the equations presented in this manuscript will play its role in the development of Finsler spacetimes to find solutions of the Finslerian Einstein equations (see for exampe \cite{HPV19,JSV21}) and that will motivate the development of the theory of Finslerian submersions as  in the case of  classical Riemannian submersions (see for example \cite{FaIaPa04}). They will also play an important role in the theory of Finsler foliations \cite{AAD19,AAJ19,AEi22} and isoparametric functions \cite{HCSR21,Xu18}, which have already deserved some attention.  Finsler submersions can also be useful in the study of some types of homogeneous Finsler manifolds \cite{XuZh18}. 

It is well-known that Finsler Geometry entails cumbersome computations in coordinates. To overcome this problem we have used the anisotropic tensor calculus  recently  developed in \cite{Jav19,Jav20}. The main idea is to 
perform all the computations with the anisotropic connections and when necessary, to fix a local vector field $V$ and to work with an associated affine connection $\nabla^V$. There are some priveleged choices of $V$ which economize computations. Moreover, it turns out that there is no need to use expressions in coordinates. The paper is organized as follows. In Section \ref{pseudo-Finsler}, we introduce the concept of pseudo-Finsler metric and its main features and we also briefly describe how the anisotropic tensor calculus works. In Section \ref{pseudo-sub}, we introduce the definition of pseudo-Finsler submersion and prove several basic properties, as for example that when one fixes a horizontal vector $v$, then the  differential map of $\sigma$ with the fundamental tensor $g_v$ has a structure of pseudo-Riemannian submersion (see Lemma \ref{Lemma1}). This result was ennunciated in \cite[Prop. 2.2]{2001AlDur}. We also prove that when the pseudo-Finsler metrics are defined in the whole tangent bundle, then there exists a unique horizontal lift of every non-zero vector in the base. Finally, we compute an expresion which relates the Cartan tensors of the total space and the base  being  the first one evaluated in at least two $g_v$-horizontal vectors. In Section \ref{fundTA}, we introduce the fundamental tensors $\T$ and $\At$ and study their properties. One of the main differences with the Riemmannian case is that now the horizontal part of any vector field is an anisotropic vector field (see Section \ref{pseudo-Finsler} for its definition), this means that for every  admissible  vector  $v$ there is a different horizontal part and we will have to deal with fiber derivatives of the horizontal and  vertical  parts, which are computed in Proposition \ref{dotPartialTopProposition}. The properties of $\T$ and $\At$ are very similar to those of their classical versions, with some exceptions. For example, the tensor $\At$  has  a much more involved expression in terms of the Lie Bracket (see Proposition \ref{TAProposition2}). To compute the Gauss formula and its dual, we need to introduce two anisotropic tensors $\hat Q$ and $\tilde Q$, because as it was known from the theory of submanifolds, the Chern connection of the fibers does not coincide with the tangent part of the connection of the ambient space. A similar situation happens with the horizontal part of  the Chern connection applied to $g_V$-horizontal vector fields for some horizontal vector field $V$  (see Proposition \ref{VXYhor}). In the subsection \ref{covderTA}, the formulas for the covariant derivatives of $\T$ and $\At$ are obtained, this time matching the classical version except for the cyclic formula in \cite[Lemma 7]{1966ONeill}. In Section \ref{fundeq}, we finally obtain the fundamental equations for the curvature tensor. As a first step, we introduce the vertical and horizontal curvature tensors in Definition \ref{RTopBotDefinition}. Then we decompose the curvature tensor using these horizontal and vertical ones and the tensors $\T$ and $\At$ in Theorem \ref{Unified} and  obtain  more specific formulas computing all the involved terms when the curvature tensor is evaluated in vertical and $g_v$-horizontal vectors in Corollary \ref{fundeqgeneral}. When $v$ is vertical, we can obtain a generalized Gauss equation as in the case of submanifolds, and when $v$ is horizontal a Dual Gauss equation (see Theorem \ref{gaussanddual}). Remarkably, Codazzi and its dual hold for an arbitrary admissible $v$.  As an application, we obtain formulas for the flag curvature with arbitrary flagpole of the ambient space in Corollary \ref{flagcurgen}. Moreover, we also recover the O'Neill formulas (see \cite[Corollary 1]{1966ONeill}) for the flag curvature when the flagpole is horizontal in Corollary \ref{horflagpole}, while the formula for a vertical flagpole is slightly different (see Corollary \ref{verflagpole}).  In the last section \ref{totallygeo}, we study Finslerian submersions with totally geodesic fibers. It is well-known that in the Riemannian case, this condition implies that the fibers are isometric in a natural way. In the Finslerian setup we need to add a further condition, namely, the horizontal regularity in Definition \ref{horreg}, to obtain these isometries. Then under this condition, we prove that a Finsler submersion has totally geodesic fibers if and only if it is an associated fiber bundle of a principle bundle, Theorem \ref{totgeoth}.
\section{Pseudo-Finsler manifolds}\label{pseudo-Finsler}

Given a manifold $\M$, we will denote by $\pi\colon T\M\to\M$ the natural projection from its tangent bundle $T\M$.
\begin{definition}
Let $\Ad\subset T\M\setminus\lbrace 0\rbrace$ be an open subset invariant under positive scalar multiplication. We say that a smooth function $L\colon\Ad\to\R$ is a pseudo-Finsler metric when it satisfies:
\begin{enumerate}[(i)]
\item positive homogeneity of degree $2$,  namely, $L(\lambda v)=\lambda^2L(v)$ $\forall\lambda>0$ and $v\in \Ad$, 
\item non-degeneracy of its associated fundamental tensor $g_{ v}$ defined for each $v\in\Ad$ and $e,h\in T_{\pi(v)}\M$ by
\begin{equation*}
g_v(e,h) = \tfrac{1}{2} \left. \tfrac{\partial^2 L(v+se+th)}{\partial s\partial t} \right\vert_{s,t=0}
\,.
\end{equation*}
\end{enumerate}
 Moreover, when $\Ad=T\M\setminus \bf 0$ and $g_v$ is always positive definite, we will say that $L$ is a Finsler metric, and it is common to work with the positive homogeneous function of degree one $F=\sqrt{L}$ rather than with $L$. 
\end{definition}
 Given a pseudo-Finsler metric, we will say that a vector $v$ is admissible if $v\in \Ad$. Moreover, given an open subset $\Omega\subset \M$, we will say that a vector field $V\in\mathfrak{X}( \Omega)$ is admissible if $V_p\in \Ad$ for all $p\in\Omega$. We will say that $V$ is locally admissible, if it is admissible in a certain neighborhood $\Omega$ of a given point. 
The  Cartan tensor $C$ associated with $L$ for each $v\in\Ad$ and $b,e,h\in T_{\pi(v)}\M$  is defined  by
\begin{equation*}
C_v(b,e,h) = \tfrac{1}{2} \left. \tfrac{\partial^2 g_{v+tb}(e,h)}{\partial t} \right\vert_{t=0} = \tfrac{1}{4} \left. \tfrac{\partial^3 L(v+rb+se+th)}{\partial r\partial s\partial t} \right\vert_{r,s,t=0}
\,.
\end{equation*}
 Observe that the Cartan tensor is symmetric and, by homogeneity, $C_v(v,\cdot,\cdot)=0$ for every $v\in\Ad$. 

Let us denote by $\mathfrak{X}(\M)$ and $\mathfrak{X}^*(\M)$ the set of (smooth) vector fields and one-forms on $\M$, respectively.  We will also denote by $T^*\M$ the cotangent bundle of $\M$ and  by $\overset{r)}\bigotimes T\M\otimes \overset{s)}\bigotimes T^* \M$ the classical tensor bundle over $\M$ of type $(r,s)$. Now denoting $\pi_\Ad:\Ad\rightarrow \M$ the restriction of the natural projection $\pi$, we will consider the pullbacked bundle $\pi_\Ad^*(\overset{r)}\bigotimes T\M\otimes \overset{s)}\bigotimes T^* \M)$. We will say that a smooth section of this bundle $T:\Ad\rightarrow \pi_\Ad^*(\overset{r)}\bigotimes T\M\otimes \overset{s)}\bigotimes T^* \M)$ is an $\Ad$-anisotropic tensor field. When the subset $\Ad$ is obvious from the context, we will say simply ``anisotropic tensor''. Observe that for every $v\in \Ad$, $T_v$ (the evaluation of $T$ at $v\in\Ad$) is a multilinear map
\[T_v:(T^*_{p}M)^r\times (T_{p}M)^s\rightarrow \R,\]	
with $p=\pi(v)$. In particular, the fundamental tensor $g$ and the Cartan tensor $C$ of a pseudo-Finsler metric are  $\Ad$-anisotropic tensors of type $(0,2)$ and $(0,3)$, respectively.  We will denote by $\mathcal{T}^r_s(\M_\Ad)$ the space of $\Ad$-anisotropic tensor fields. As in the classical case, $\mathcal{T}^1_0(\M_\Ad)$ can be identified with $\Ad$-anisotropic vector fields $\mathcal X$,  $\Ad\ni v\rightarrow {\mathcal X}_v\in T_{\pi(v)}\M$.  Moreover, if 
${\mathcal F}(\Ad)$ is the space of smooth real functions on $\Ad$,
 $\Ad$-anisotropic tensor fields can be described equivalently as ${\mathcal F}(\Ad)$-multilinear maps
\begin{equation*}
T\colon {\mathcal T}^0_1(\M_\Ad)^r\times   {\mathcal T}^1_0(\M_\Ad)^s \rightarrow {\mathcal F}(\Ad)
\end{equation*}
very much as in the classical case, a tensor field can be thought of as a ${\mathcal F}(\M)$-multilinear maps.
\begin{definition}
An ($\Ad$-)anisotropic connection is a  map
\begin{equation*}
 \nabla\colon \mathfrak X(\M)\times\mathfrak X(\M) \rightarrow \mathcal{T}^1_0(\M_\Ad)\,,\quad (E,H)\rightarrow \nabla_EH
\end{equation*}
satisfying for all vector fields $B,E,H$ and smooth  real  function $f$  on $\M$
\begin{enumerate}[(i)]
\item $\nabla_E(H+B)=\nabla_EH+\nabla_EB$,
\item $\nabla_E(fH)=(E(f)\cdot Y)\circ\pi_\Ad+(f\circ \pi_\Ad)\nabla_EH$,
\item $\nabla_{E+B}H=\nabla_EH+\nabla_BH$,
\item $\nabla_{fE}H=(f\circ\pi_\Ad)\nabla_EH$.
\end{enumerate}
\end{definition}
We will use the notation $\nabla^v_EH:=(\nabla_EH)_v$. Moreover, 
since the expression $\nabla^v_EH$ depends only on the vector $e=E_{\pi(v)}$ rather than  on  its extension $E$, we shall prefer to write $\nabla^v_eH$.
Given a vector field $V$ on an open subset $\Omega$ of $\M$ which is $\Ad$-admissible, namely, such that $V_p\in \Ad$ for all $p\in\Omega$, we can consider an affine connection $\nabla^V$ on $\Omega$ defined for each $E,H\in {\mathfrak X}(\M)$ by $(\nabla^V_EH)_p:=(\nabla_EH)_{V_p}\equiv \nabla^{V_p}_EH $. Associated with the anisotropic connection there is a covariant derivative along curves which depends on the choice of an admissible vector field along the curve (with values in $\Ad$ at every instant). This covariant derivative is constructed locally with the Christoffel symbols of the anisotropic connection and for a curve $\gamma:[a,b]\rightarrow \M$ and an admissible vector field $W\in  {\mathfrak X}(\gamma)$,  it will be denoted by $D^W_\gamma: {\mathfrak X}(\gamma)\rightarrow  {\mathfrak X}(\gamma)$, $X\rightarrow D^W_\gamma X$. When the curve $\gamma$ is regular at $p=\gamma(t)$, it is given by $(\nabla^{\tilde W}_V\tilde X)_p$, considering vector fields $\tilde W,V,\tilde X$ which are local extensions of $W,\dot\gamma,X$, respectively (see \cite[Section 2.2]{Jav20} for more details).  Given an $\Ad$-anisotropic tensor $T$, we can also define a classical tensor of the same type $T_V$ as $(T_V)_p=T_{V_p}$.  Departing from an anisotropic connection it is possible to define the covariant derivative of any anisotropic tensor as it was shown in \cite{Jav19}. These covariant derivatives can be obtained by the classical product rule, and they can also be expressed in terms of $\nabla^V$ with the help of fiber derivatives. If $T\in {\mathcal T}^r_s(\M_\Ad)$, then its fiber derivative  $\dot\partial T\in {\mathcal T}^r_{s+1}(\M_\Ad)$  is defined as
\begin{equation*}
(\dot\partial_e T)_v
=
\left.\tfrac{\partial T_{v+te}}{\partial t} \right\vert_{t=0}
\end{equation*}
where $e\in T_{\pi(v)}\M$  is the $s+1$-covariant component of $\dot\partial T$ and sometimes we will put it in its natural place.  
 A key observation is that anisotropic connections allow us to define derivations of functions $\mathcal F(\Ad)$. Indeed, for any $v\in \Ad$ and any choice of local admissible extension $V$ of $v$, $\nabla_Ef$ is smooth function on $\Ad$ given by 
\begin{equation}\label{nablaf}
(\nabla_E f)(v)=E_{\pi(v)}(f(V))-\dot\partial_{\nabla^v_EV} f.
\end{equation}
It turns out that last expression is independent of the choice of $V$  (see \cite[Lemma 9]{Jav19}). 
Moreover, this implies that given an anisotropic tensor $T \in \mathcal{T}^r_s(\M_\Ad)$,  vector fields  $X_1,\ldots,X_r\in {\mathfrak X}(M)$ and  one-forms  $\theta^1,\ldots,\theta^s\in {\mathfrak X}^*(M)$,  the covariant derivative
of $T$ with respect to $E$, denoted $\nabla_ET,$ is given  by
\begin{align}
\nabla_ET(\theta^1,\ldots,\theta^s,X_1,\ldots,X_s) = &\nabla_E(T(\theta^1,\ldots,\theta^r,X_1,\ldots,X_s))\nonumber\\&-
\sum_{i=1}^r (T(\theta^1,\ldots,\nabla^v_E\theta^i,\ldots,\theta^r,X_1,\ldots,X_s)
\nonumber\\&-\sum_{j=1}^s (T(\theta^1,\ldots,\theta^r,X_1,\ldots,\nabla^v_EX_j,\ldots,X_s)\label{covder}
\end{align}
while  the first term on the right, which is  the anisotropic  covariant derivative of $v\rightarrow T_v(\theta^1,\ldots,\theta^r,X_1,\ldots,X_s)$, can be expressed as
\begin{multline}\label{fiberder}
\nabla_E (T(\theta^1,\ldots,\theta^r,X_1,\ldots,X_s))(v)=E_{\pi(v)}(T_V(\theta^1,\ldots,\theta^r,X_1,\ldots,X_s))\\-(\dot\partial_{\nabla^v_EV} T)_v(\theta^1,\ldots,\theta^r,X_1,\ldots,X_s)).
\end{multline}
Observe that the anisotropic connection can also be applied to anisotropic vector fields. Indeed, if ${\mathcal X}\in{\mathcal T}^1_0(\M_\Ad)$, then
\begin{equation}\label{nablaP}
\nabla^v_e{\mathcal X}=\nabla^v_e {\mathcal X}_V-\dot\partial {\mathcal X}_v(\nabla^v_eV), 
\end{equation}
 for any choice of admissible local extension $V$ of $v$. 
We may define then the associated curvature tensor $R_v:{\mathfrak X}(\M)\times {\mathfrak X}(\M)\times {\mathfrak X}(\M)\rightarrow T_{\pi(v)}\M $ by
\begin{equation*}
R_v(E,H)B
=
\nabla^v_E \nabla_HB
-
\nabla^v_H \nabla_EB
-
\nabla^v_{[E,H]}B.
\end{equation*}
Then $R$ determines an anisotropic tensor field of type $(1,3)$, with the standard identifications (see \cite{Jav19,Jav20}). Moreover, this tensor can be expressed in terms of the curvature tensor $R^V$ of the affine connection $\nabla^V$ for any local admissible extension  $V$ of $v$:
\begin{equation*}
R_V(E,H)B
=
R^V(E,H)B
-
P_V(H,B,\nabla^V_EV)
+
P_V(E,B,\nabla^V_HV)
\,,
\end{equation*}
 where $P$ is the fiber derivative of $\nabla$ defined as
\begin{equation}\label{nablader}
P_v(E,H,B)
=
\left. \tfrac{\partial}{\partial t} \nabla^{v+tB\vert_{\pi(v)}}_EH \right\vert_{t=0}
\,,
\end{equation}
 for $E,H$ and $B$ arbitrary vector fields on $\M$ (see \cite[Prop. 2.5]{Jav20}). 
Let us introduce a privileged anisotropic connection associated with a pseudo-Finsler manifold. 
\begin{definition}
The Levi-Civita--Chern connection of $(\M,L)$ a pseudo-Finsler manifold is the unique torsion-free $g$-compatible anisotropic connection $\nabla$, satisfying namely for each admissible $v\in\Ad$ and vector fields $E,H\in\mathfrak{X}(\M)$
\begin{enumerate}[(i)]
\item $\nabla^v_EH-\nabla^v_HE-[E,H]\vert_{\pi(v)}=0$,
\item $\nabla g = 0$.
\end{enumerate}
\end{definition}
By our definition of covariant derivation of anisotropic tensors, the metric-preser\-ving condition $\nabla g = 0$ is also referred to as the almost-compatibility condition due to the equivalent identity for each $v\in\Ad$ with locally admissible extension $V$ and $e,h,b\in T_{\pi(v)}\M$ with respective extensions $E$, $H$ and $B$
\begin{equation*} e (g_V(H,B)) \vert_{\pi(v)}
=
g_v(\nabla^v_eH,B)
+
g_v(H,\nabla^v_eB)
+
2C_v(h,b,\nabla^v_eV)
\, ,
\end{equation*}
 observing that the fiber derivative of $g$ is $2C$. 
Uniqueness can be easily checked from the associated Koszul formula
\begin{multline}\label{koszul}
2g_v(\nabla^v_eH,B)
=
e(g_V(B,H))+h(g_V(E,B))-b (g_V(E,H)) 
\\
+
g_v([B,E],h)+g_v(e,[B,H])+g_v([E,H],B)
\\
-
2C_v(b,h,\nabla^v_eV)
-
2C_v(e,b,\nabla^v_hV)
+
2C_v(e,h,\nabla^v_bV)
\end{multline}
(see for example \cite[Th. 4]{JSV22}). 
\begin{lemma}
Let $\mathcal{X}$ be an anisotropic vector field, $v\in\Ad$ and $e,h\in T_{\pi(v)}\M$.   If $\nabla$ is an anisotropic connection and $P$ its vertical derivative defined in \eqref{nablader},  then 
\begin{equation}\label{dotPartialNablaMathcalX}
(\dot\partial (\nabla_e \mathcal{X}))_v(h)=P_v(e,\mathcal{X}_v,h)+(\nabla_e(\dot\partial \mathcal{X}))_v(h)-(\dot\partial \mathcal{X})_v(P_v(e,v,h)).
\end{equation}
\end{lemma}
\begin{proof}
Consider a locally admissible extension $V$ of $v$ and respective extensions $E$ and $H$ of $e$ and $h$. 
By the chain rule,  recalling \eqref{nablaP}   and observing that the second fiber derivative $\dot\partial^2 {\mathcal X}$ is symmetric, 
\begin{align*}
(\dot\partial (\nabla_e \mathcal{X}))_v(h)
= &
\left.\tfrac{\partial}{\partial t} \nabla^{v+th}_e\mathcal{X}\right\vert_{t=0}
\\
= &
\left.\tfrac{\partial}{\partial t}\left( \nabla^{v+th}_e\mathcal{X}_{V+tH}
-
(\dot\partial\mathcal{X})_{v+th}(\nabla^{v+th}_e(V+tH))\right)\right\vert_{t=0}
\\
=
&
\left.\tfrac{\partial}{\partial t} \nabla^{v+th}_e\mathcal{X}_{V} \right\vert_{t=0}
+
\left.\tfrac{\partial}{\partial t} \nabla^v_e\mathcal{X}_{V+tH} \right\vert_{t=0}
-
\left.\tfrac{\partial}{\partial t} (\dot\partial\mathcal{X})_{v+th}(\nabla^v_eV) \right\vert_{t=0}
\\
& \quad
-
\left.\tfrac{\partial}{\partial t} (\dot\partial\mathcal{X})_v(\nabla^{v+th}_eV) \right\vert_{t=0}
-
\left.\tfrac{\partial}{\partial t} (\dot\partial\mathcal{X})_v(\nabla^v_eV+t\nabla^v_eH)) \right\vert_{t=0}
\\
=
&
P_v(e,\mathcal{X}_v,h)
+
\nabla^v_e((\dot\partial \mathcal{X})_V(H))
-
(\dot\partial^2\mathcal{X})_v(\nabla^v_eV,h)
\\
& \quad
-
(\dot\partial\mathcal{X})_v(P_v(e,v,h))
-
(\dot\partial\mathcal{X})_v(\nabla^v_eH)
\\
=
&
P_v(e,\mathcal{X}_v,h)
+
(\nabla_e(\dot\partial\mathcal{X}))_v(h)
-
(\dot\partial \mathcal{X})_v((P_v(e,v,h)).
\end{align*}
\end{proof}
\section{Pseudo-Finsler submersions}\label{pseudo-sub}

Let $ (\M,L) $ and $ (\B,\LB) $ be pseudo-Finsler manifolds with $L:\Ad\subset T\M\rightarrow\R$ and $\tilde L:\tilde\Ad\subset T\B\rightarrow\R$, and $ \sigma \colon \M \to \B $ a submersion, namely, a surjective map with surjective differential. In this situation, it is well-known that the fibers $\fiber_q=\sigma^{-1}(q)$ are submanifolds of $\M$ for every $q\in\B$ and we will denote by $ \ver_p := \ker\mathrm d\sigma\vert_p $ the vertical subspace, namely, the set of vectors tangent to the fibers, which will be called {\em vertical vectors} from now on.

As a counterpart to the vertical vectors, we define the {\em horizontal vectors} in the following way: an admissible vector $v\in\Ad_p:=\Ad\cap T_p\M$ is horizontal if and only if it is $g_v$-orthogonal to $\ver_p$, namely, $g_v(v,w)=0$ for each vertical vector $w\in\ver_p$. The subset of horizontal vectors of $T_p\M$ will be denoted by $ \hor_p $, and has the structure of a (not necessarily linear) submanifold.
\begin{lemma}\label{horsub}
Given a submersion $ \sigma \colon \M \to \B $ between pseudo-Finsler manifolds $ (\M,L) $ and $ (\B,\LB) $,
$\hor_p$ is a submanifold of $T_p\M$ at each $p\in\M$.  Moreover, $\hor=\cup_{p\in\M} \hor_p$ is also a submanifold of $T\M$. 
\end{lemma}
\begin{proof}
Let $e_1,\ldots,e_r$ be a basis of $\ver_p$ and define the map $f:\Ad_p\rightarrow \R^r$, given by 
$f(v)=(g_v(v,e_1),\ldots,g_v(v,e_r))$. It turns out that $\dif f_v(u)=(g_v(u,e_1),\ldots,g_v(u,e_r))$, and then, the kernel of $\dif f_v$ is given by the $g_v$-orthogonal vectors to the fiber $\ver_p$. As $g_v$ is a scalar product, the dimension of ${\rm ker}(\dif f_v)$ is constantly equal to $n-r$, being $n=\dim\M$ (see \cite[Lemma 2.22]{1983ONeill}). This implies that $f$ is a submersion, and $\hor_p=f^{-1}(0)$ a smooth submanifold.  To prove that $\hor$ is a submanifold of $T\M$, consider an adapted chart to the submersion to obtain a frame $E_1,\ldots, E_r$ of vertical vector fields in a neighborhood $U$of $p$. Then define $f: \Ad\cap TU\rightarrow \R$ as above using now the frame $E_1,\ldots, E_r$ instead of $e_1,\ldots,e_r$  and follow the same steps as before. 
\end{proof}
\begin{definition}
A submersion $\sigma \colon (\M,L) \to ( \B,\LB) $ between pseudo-Finsler manifolds is called pseudo-Finsler if  both of the following conditions are satisfied:
\begin{enumerate}[(i)]
\item the submersion fibers are non-degenerate pseudo-Finsler submanifolds of the total space $(\M,L)$,  namely, the restriction of $L$ to the tangent space to the fibers provides a pseudo-Finsler metric on the fibers, 
\item $\dif\sigma$ preserves the length of horizontal vectors, namely, $L(v)=\LB(\dif\sigma(v))$ holds for each horizontal $v$.
\end{enumerate}
\end{definition}
Let us see that a pseudo-Finsler submersion induces a pseudo-Riemannian submersion at each tangent space. In the following, we will use the notation $\tilde v=\dif\sigma(v)$,   $\tilde x=\dif\sigma(x),\ldots$ and will use a tilde to denote the elements associated with $\tilde L$, namely, $\tilde g$ denotes its fundamental tensor, $\tilde C$ its Cartan tensor and so on. Moreover, we will denote by $\Sigma$, the indicatrix of $L$, namely, the admissible vectors $v\in \Ad$ with $L(v)=1$. The indicatrix of $\LB$ will be denoted by $\tilde\Sigma$. 
\begin{lemma}\label{Lemma1}
For each horizontal $ v $, $ \dif\sigma\vert_p:T_p\M\to T_{\sigma(p)}\B $ is a pseudo-Riemannian submersion from $ (T_p\M,g_v)$ to $(T_{\sigma(p)}\B,\tilde g_{\tilde v})$, namely,
\begin{equation}\label{gv1}
g_v(x,y) = \tilde g_{\tilde v}(\tilde x,\tilde y)
\end{equation}
for all $x,y$ $g_v$-orthogonal to the vertical vectors at $T_p\M$. In particular,
\begin{equation}\label{gv2}
g_v(x,e) = \tilde g_{\tilde v}(\tilde x,\tilde e)
\end{equation}
for every $e\in T_p\M$ with $\tilde{e} = \dif\sigma(e)$.
\end{lemma}
\begin{proof}
As lightlike vectors are in the closure of $\{v\in\Ad: L(v)\neq0\}$, it suffices to prove \eqref{gv1} under the condition $L(v)\neq0$ and extend it by continuity to the lightcone. By $0$-homogeneity of $ g $, we can furthermore assume that $L(v)=\pm 1$. By definition of a pseudo-Finsler submersion, $g_v(v,v)=\tilde g_{\tilde v}(\tilde v,\tilde v)$ for every horizontal $v$, and the vectors tangent to the indicatrix $\Sigma$ are therefore precisely those $g_v$-orthogonal to $v$. As a consequence, it suffices to prove that for $x,y\in T_v\Sigma$
\begin{equation}\label{igualsenza}
g_v(x,y) = \tilde{g}_{\tilde{v}}(\tilde{x},\tilde{y})
\,.
\end{equation}
Since the result is local, considering an adapted chart, we can assume that the submersion is a canonical projection $\sigma:\R^n\rightarrow \R^m$. Consider now $\Lambda_p=\Sigma_p\cap\hor_p$, submanifold of $\Sigma_p$ by Lemma \ref{horsub} as it is the intersection of two transversal smooth submanifolds. Observe that $\dif \sigma_p(\Lambda_p)\subset \tilde \Sigma_{\sigma(p)}$, and therefore $\Lambda_p\subset\dif\sigma\vert_p^{-1}(\tilde\Sigma_{\sigma(p)})$. Moreover, as $\dif\sigma_p$ is a submersion whose fibers are affines subspaces parallel to $\ver_p$, the second fundamental form with respect to the trivial connection (the metric connection of any Euclidean metric on $T_p\M$) of $\C_p=\dif\sigma\vert_p^{-1}(\tilde\Sigma_{\sigma(p)})$ has in its kernel all the vertical vectors. Then,  recalling  \cite[Eq. (2.5)]{JavSan14}) for the relationship between the fundamental tensor and the second fundamental form of $\Sigma$,  \eqref{igualsenza} follows from the following observations:
\begin{enumerate}[(i)]
\item the second fundamental form of $\Sigma_p$ at $v$ having as transversal vector the vector $-v$ coincides with the restriction of $g_v$ to $T_v\Sigma_p\times T_v\Sigma_p$,
\item similarly the second fundamental form of $\tilde\Sigma_{\sigma(p)}$ at $\tilde v$ having as transversal vector $-\tilde{v}$ coincides with the restriction of $\tilde g_{\tilde v}$ to $T_{\tilde v}\tilde\Sigma_{\sigma(p)}\times T_{\tilde v}\tilde\Sigma_{\sigma(p)}$,
\item $\tilde \Sigma_{\sigma(p)}$ can be identified with $\C_p\cap\R^m$ (or, to be precise, $\C_p\cap(\lbrace 0 \rbrace^{n-m}\times\R^m)$),
\item $\C_p$ is invariant under vertical translations,
\item $v-\tilde v$ is tangent to $\C_p$ (it is indeed vertical), and thus the two second fundamental forms of $\C_p$ along the fiber $(\dif\sigma)_p^{-1}(\tilde v)$ having as respective transversal vectors $-v$ and $-\tilde v$ coincide.
\end{enumerate} 
\end{proof}
Recall that the Legendre map $\leg_p\colon T_p\M\to T_p\M^*$ of a pseudo-Finsler metric is defined on each admissible $v$ as
$\leg(v)=g_v(v,\cdot)$. If the pseudo-Finsler metric is defined in the whole tangent space, then $\leg$ is bijective whenever $\dim\M\geq 3$ (see \cite{Min15,RuSu15}). 
\begin{lemma}\label{legendrelemma}
If the Legendre map $\leg$ at $p\in\M$ is injective (bijective), then each  $\tilde v\in \tilde\Ad\cap T_{\sigma(p)}\B$  admits at most one (exactly one) horizontal vector $v\in T_p\M$ satisfying $\dif\sigma(v)=\tilde v$.   
\end{lemma}
\begin{proof}
For any two horizontal vectors $v_1,v_2\in T_p\M$ projecting onto $\tilde v$, let us prove $\leg_p(v_1)=\leg_p(v_2)$. Observe that $\leg_p(v_1)$ and $\leg_p(v_2)$ have the same kernel $\ver_p$, such that each $g_{v_1}$-horizontal $x_1$ admits a vertical vector $w$ such that $x_2=x_1+w$ is $g_{v_2}$-horizontal. Denoting by $\tilde{x}$ their shared projection, by Lemma \ref{Lemma1},
\begin{equation*}
\leg_p(v_1)(x_1)=g_{v_1}(v_1,x_1)=\tilde g_{\tilde v}(\tilde v,\tilde x)=g_{v_2}(v_2,x_2)=g_{v_2}(v_2,x_1)=\leg_p(v_2)(x_1)
\,,
\end{equation*}
or specifically $\leg_p(v_1)=\leg_p(v_2)$, which implies that $v_1=v_2$ by injectivity of $\leg_p$. Let us now prove the existence of the horizontal lift when $\leg_p$ is invertible.   Fixing an admissible vector $\tilde v$ at $\sigma(p)$, and a co-vector $\omega$ at $p$ with kernel $\ver_p$ that acts on a transversal subspace as $\tilde g_{\tilde v}(\tilde v,\cdot)$ acts on its projection, it is straightforward to check that $v=(\leg_p)^{-1}(\omega)$ is horizontal and projects onto $\tilde v$. 
\end{proof}
\begin{remark}
 Throughout  this paper, we will need to consider only a local horizontal extension $V$ of a horizontal vector $v$. As the  Legendre  transform is locally invertible, the existence of these local horizontal extensions is always guaranteed for every horizontal $v$ by  a local application of  Lemma \ref{legendrelemma}.   Indeed, consider an extension $\tilde V$ of the projection $\tilde v$ and a neighborhood of $v$ in $\Ad$ where the Legendre transform is injective. Then the one-form constructed in the proof of the last lemma is on the image of the Legendre transform for a small enough neighborhood of $g_v(v,\cdot)$ and its inverse by $\leg$ gives a horizontal vector field which projects onto $\tilde V$. 
\end{remark}
\begin{corollary}
For each horizontal $v$, $g_v$-horizontal $x$ and $y$ and arbitrary $e$ (at $p$),
\begin{equation}
C_v(x,y,e) = \tilde C_{\tilde v}(\tilde x,\tilde y,\tilde e)
- \tfrac{1}{2}g_v(\mathrm{I\!I}^\hor_v(x,y),e)
\,,
\end{equation}
where $\mathrm{I\!I}^\hor_v$ is the second fundamental form of $\hor_p$ as a submanifold of $(T_p \M,g_v) $.
\end{corollary}
\begin{proof}
Observe that the space tangent to $\hor_p$ at $v$ is the set of $g_v$-horizontal vectors (see the proof of Lemma \ref{horsub} and observe that this space is the kernel of  $\dif_vf$).  Moreover, consider any smooth curve $\gamma$ of $\hor_p$ with $\gamma(0)=v$ and $\dot\gamma(0)=x$, and vector fields $Y$ and $Z$ along $\gamma$ tangent to $\hor_p$ with $Y(0)=y$ and $Z(0)=e-w$ for some (constant) vertical $w$, such that $Y$ and $Z$ are $g_\gamma$-orthogonal to $\ver_p$ at each point of $\gamma$. By \eqref{gv1} we have
\begin{equation*}
g_\gamma(Y,Z+w) = \tilde g_{\tilde\gamma}(\tilde Y,\tilde Z)
\,.
\end{equation*}
By differentiating both sides,
\begin{equation*}
2C_\gamma(Y,Z+w,\dot\gamma) + g_\gamma(\dot Y,Z+w) + g_\gamma(Y,\dot Z)
=
2\tilde C_{\tilde\gamma}(\tilde Y,\tilde Z,\dot{\tilde\gamma}) + \tilde g_{\tilde\gamma}(\dot{\tilde Y},\tilde Z) + \tilde g_{\tilde\gamma}(\tilde Y,\dot{\tilde Z}).
\end{equation*}
By \eqref{gv1}, and as $\dif\sigma(\dot X)=\dot{\tilde X}$ and $\dif\sigma(\dot Y)=\dot{\tilde Y}$ by linearity of $\dif\sigma$, this cancels to
\begin{equation*}
2C_v(y,z+w,x) + g_v(\dot Y,w) = 2\tilde C_{\tilde v}(\tilde y,\tilde z,\tilde x)
\,.
\end{equation*}
To conclude, the second fundamental form $\mathrm{I\!I}^\hor_v(x,y)$ is the vertical part of the covariant differentiation of $Y$ along $x$.
\end{proof}

\section{The fundamental tensors T and A}\label{fundTA}

To introduce the fundamental tensors of a submersion we will need first to define the vertical and horizontal parts of each vector in $T\M$.  In the pseudo-Finsler case, we will need a further condition to do this decomposition, namely, the vertical space must be non-degenerate for all $v\in \Ad$. By the definition of pseudo-Finsler submersion, this will be true for vertical and horizontal vectors, and we can reduce the domain where $L$ is defined to ensure this condition, which does always hold  when $g_v$ is positive definite. As a consequence $\Ad$ could be non-connected.  Then if $e\in T_{\pi(v)}\M$, one has the $g_v$-decomposition
\[e= e^\top_v+e^\bot_v\]
with $e^\top_v$ vertical and $e^\bot_v$ $g_v$-orthogonal to $\ver_p$. In particular, given a vector field $X\in{\mathfrak X}(\M)$, we can define $X^\top$ and $X^\bot$ as the anisotropic vector fields such that the evaluation at $v\in \Ad_p$ is given by $X^\top_v=(X_p)^\top_v$ and $X^\bot_v=(X_p)^\bot_v$, respectively. 
\begin{lemma}\label{XT}
Given a pseudo-Finsler submersion and a vector field $X\in{\mathfrak X}(\M)$, then $X^\top,X^\bot\in {\mathcal T}^1_0(\M)$.
\end{lemma}
\begin{proof}
It suffices to prove the smoothness of $X^\top$. Consider in a neighborhood of $p\in\M$ coordinates $(x_1,\ldots,x_n)$ adapted to the submersion  and $v\in \Ad_p$.  Some constant coefficient combination of the induced vector fields $\frac{\partial}{\partial x^1},\ldots,\frac{\partial}{\partial x^n}$ produces a non-degenerate reference frame $e_1,\ldots,e_n$ such that the first vectors $e_1,\ldots,e_r$ form a reference frame for the vertical space in a neighborhood $U$ of $p$, 
 and for all $1\leq s\leq n$, $e_1,\ldots,e_s$ spans a $g_w$ non-degenerate subspace for $w$ in a neighborhood of $v$ 
 (this can be done by perturbing the frame of partial vector fields   by reducing the neighborhood if necessary). Under these conditions, we can apply the Gram-Schmidt process to obtain a $g_w$-orthonormal basis $e_1(w),\ldots,e_n(w)$ with a smooth dependence on $w$ in some neighborhood of $v$, such that $e_1(w),\ldots, e_r(w)$ is a basis of the vertical subspace. Locally, $X^\top_w = \sum_{i=1}^r \frac{g_w(e_i(w),X)}{g_w(e_i(w),e_i(w))} e_i(w)$, smooth in $w$.
\end{proof}
Observe that Lemma \ref{XT} can be applied to one single vector $e\in T_p\M$, obtaining a smooth map $e^\top:
\Ad_p\rightarrow T_{\pi(e)}\M$.

The letter $p$ will always denote a point of the manifold $\M$, the letters $s,u,w$ will always denote vertical vectors, namely elements of $\ver_p$, $v$ an admissible vector, namely an element of $\Ad_p$, and $x,y,z$ $ g_v $-horizontal vectors of $T_p\M$. The  capital  letters $S,U,W$ will always denote vertical vector fields, $V$ a locally admissible extension of $v$, and $X,Y,Z$ locally $g_V$-horizontal vector fields, preferably projectable onto some vector fields $\tilde X,\tilde Y,\tilde Z$ of the base manifold in the sense that $\dif\sigma\cdot X=\tilde X\circ\sigma$ and so on.

The covariant derivative $\nabla^v_h E$ of an arbitrary extension $E$ of $e$ along each vector $h$ at $p$ defines the $1$-form $\nabla^v E \colon \mapsto \nabla^v_h E$ on $T_p\M$ which splits as
\begin{multline*}
\nabla^v E
=
(\nabla^v (E^\top+E^\bot))^\top_v+(\nabla^v (E^\top+E^\bot))^\bot_v
\\
=
(\nabla^v E^\top)^\top_v
+
(\nabla^v E^\bot)^\top_v
+
(\nabla^v E^\top)^\bot_v
+
(\nabla^v E^\bot)^\bot_v
\,,
\end{multline*}
whose middle terms are independent of $E$ and define the tensorial expression
\begin{equation*}
\chi^v e
=
(\nabla^v E^\bot)^\top_v
+
(\nabla^v E^\top)^\bot_v
\,.
\end{equation*}
Recall by Lemma \ref{XT} that $ E^\top $ and $ E^\bot $ are anisotropic vector fields. Furthermore, the tensor $\chi$ is well-defined if we substitute $E$ for an anisotropic vector field $\mathcal{X}$, with the definition that $\mathcal{X}^\top \colon v \mapsto (\mathcal{X}_v)^\top_v$ and similarly $\mathcal{X}^\bot \colon v \mapsto (\mathcal{X}_v)^\bot_v$.
\begin{definition}\label{ATDefinition}
 Assume that the vertical space is $g_v$-non-degenerate for all $v\in \Ad$. Then the fundamental anisotropic tensors of a pseudo-Finsler submersion are defined as the 
$(1,2)$-anisotropic tensors $\T$ and $\At$ given by
\begin{equation*}
\T^v_b e = \chi^v_{b^\top_v}e, \quad \At^v_b e = \chi^v_{b^\perp_v}e,
\end{equation*}
for any $v\in\Ad_p$ and $e,b\in T_p\M$. Therefore $\chi=\T+\At$.
\end{definition}
We may check the following properties, which differ slightly from \cite{1966ONeill}. Note our use of the musical isomorphism to define $ C^\sharp $ as the symmetric type $ (1,2) $ tensor
determined for arbitrary vectors $ b $, $ e $ and $ h $ at $ p $ by
\begin{equation*}
g_v(b,C^\sharp_v(e,h))
=
C_v(b,e,h)
\,,
\end{equation*}
and by 
homogeneity, $ C^\sharp_v(v,\cdot) = 0 $. 
\begin{proposition}\label{dotPartialTopProposition}
For each admissible $ v $, vertical $ w $ and $ g_v $-horizontal $ x $ at $ p $,
\begin{equation*}
(\dot\partial w^\top)_v = (\dot\partial w^\bot)_v = 0
\,, \quad
(\dot\partial x^\top)_v = 2C^\sharp_v(x,\cdot)^\top_v
\,, \quad
(\dot\partial x^\bot)_v = -2C^\sharp_v(x,\cdot)^\top_v
\,.
\end{equation*}
That is to say, for an arbitrary vector $e$ at $p$,
\begin{equation}\label{vertder}
(\dot\partial e^\top)_v= 2C^\sharp_v(e^\perp_v,\cdot)^\top_v, \quad (\dot\partial e^\perp)_v= -2C^\sharp_v(e^\perp_v,\cdot)^\top_v.
\end{equation} 
\end{proposition}
\begin{proof}
The identities involving $w$ are due to $w^\top$ being identically equal to $w$ and $w^\bot$ being identically zero. As for the identities involving $x$, consider an arbitrary $ e \in T_p \M $ along with a smooth one real parameter family of $ g_{v+te} $-orthonormal bases $ (u_1(t),\ldots,u_{ r}(t)) $ of $ \ver_p $ and denote by $ (\dot{u}_1(t),\ldots,\dot{u}_r(t)) $ their derivatives with respect to $t$. Observe that by continuity, $\varepsilon_i:=g_{v+te}(u_i(t),u_i(t))$ is constantly equal to $1$ or $-1$ and it does not depend on $t$. Then
\begin{equation*}
x^\top_{v+te}=\sum_{i=1}^r \varepsilon_i g_{v+te}(u_i(t),x) u_i(t)
\,,
\end{equation*}
which differentiates at $t=0$ to
\begin{equation*}
(\dot\partial x^\top)_v(e)
=
\sum_{i=1}^r \varepsilon_i \left( 2C_v(u_i(0),x,e) u_i(0)
+
g_v(\dot u_i(0),x) u_i(0)
+
g_v(u_i(0),x) \dot{u}_i(0) \right)
\,.
\end{equation*}
As $u_i(0)$ and $\dot u_i(0)$ are vertical, the last two terms are zero and the results follow straightforwardly. The last identities \eqref{vertder} follow from the first part of the proposition and the fact that $0=\dot\partial e=\dot\partial (e^\top)_v+\dot\partial (e^\perp)_v$, and then 
$\dot\partial (e^\perp)_v=-\dot\partial (e^\top)_v$.
\end{proof}
\begin{lemma}\label{Lemma3}
	For each admissible $ v $, vertical $ w $ and $ g_v $-horizontal $ x $ at $ p $ and a vertical vector field $ U $ with $ u = U_p $, any locally admissible extension $ V $ of $v$ and locally $ g_V $-horizontal vector field $ Y $ with $ y = Y_p $,
\begin{align}\label{UTA}
\nabla^v_w U &= \T^v_w u + (\nabla^v_w U)^\top_v,& \nabla^v_x U &= \At^v_x u + (\nabla^v_x U)^\top_v,
\\ \label{YTA}
\nabla^v_w Y &= (\nabla^v_w Y)^\bot_v + \T^v_w y - 2C^\sharp_v(y,\nabla^v_w V)^\top_v, & \nabla^v_x Y &= (\nabla^v_x Y)^\bot_v + \At^v_x y - 2C^\sharp_v(y,\nabla^v_x V)^\top_v.
\end{align}
	Furthermore, if $ Y $ is projectable, 
	\begin{equation}\label{Lieconversion}
	 (\nabla^v_w Y)^\bot_v = \At^v_y w.
	\end{equation}
\end{lemma}
\begin{proof}
	From the decomposition of the covariant derivative of an arbitrary vector field $ E $ with $ e = E_p $ and by definition of the tensors $ \T $ and $ \At $
	\begin{equation*}
	\nabla^v E
	=
	(\nabla^v E^\top)^\top_v
	+
	(\T+\At)^v e
	+
	(\nabla^v E^\bot)^\bot_v
	\,,
	\end{equation*}
	we obtain for a vertical vector field $ U $ with $ u = U_p $
	\begin{equation}\label{nablaU}
	\nabla^v U
	=
	(\nabla^v U)^\top_v
	+
	(\T+\At)^v u,
	\end{equation}
	which implies \eqref{UTA}.
	Moreover, from the above decomposition, taking into account that $Y=Y^\bot_V$ and $Y^\top_V=0$, and using \eqref{nablaP} and Proposition \ref{dotPartialTopProposition}, 
	\begin{multline}\label{nablaY}
	\nabla^v Y
	=
	(\T+\At)^v y
	+
	(\nabla^v Y)^\bot_v
	-
	(\dot\partial Y^\top)_v(\nabla^v V)^\top_v
	-
	(\dot\partial Y^\bot)_v(\nabla^v V)^\bot_v
	\\
	=
	(\T+\At)^v y
	+
	(\nabla^v Y)^\bot_v
	-
	2C^\sharp_v(y,\nabla^v V)^\top_v
	\,,
	\end{multline}
	which is equivalent to \eqref{YTA}. For the last identity \eqref{Lieconversion} observe that the Lie bracket of a vertical vector field and a  projectable  vector field is vertical, and then one has that $(\nabla^v_w Y)^\bot_v=(\nabla^v_y W)^\bot_v=\At^v_yw$, being $W$ a vertical extension of $w$, where in the last equality we have used \eqref{UTA}.
\end{proof}
\begin{proposition}\label{TAProposition}
For each admissible $ v $, arbitrary $e$, $ g_v $-horizontal $ x $ and vertical $ u $ and $ w $ at $p$,
\begin{enumerate}[(i)]
\item $\T^v_e$ and $\At^v_e$ are skew-symmetric on $ T_p \M $, and map $g_v$-horizontal vectors into vertical ones and vice versa,
\item if $v$ is vertical, then $\T^v$ satisfies $\T^v_u w = \T^v_w u$,
\item if $ v $ is horizontal, then $\At^v$ satisfies $\At^v_x v = - \At^v_v x$ and in particular $\At^v_vv = 0$.
\end{enumerate}
\end{proposition}
\begin{proof}
Consider a $ g_v $-horizontal vector $ y $ at $ p $.
By \eqref{nablaY} and \eqref{nablaU}, for any local choice of an admissible extension $ V $ of $v$, vertical extension $ U $ of $u$ and $ g_V $-horizontal extension $ Y $ of $y$, it follows that
\begin{multline*}
g_v((\T+\At)^v y,u)
=
g_v(\nabla^v Y,u) + 2C_v(y,u,\nabla^v V)
=
- g_v(y,\nabla^v U)
\\
=
- g_v(y,(\T+\At)^v u)
\,,
\end{multline*}
where we have also used in the second equality that $ g_V(Y,U)=0$, and therefore $0=e(g_V(Y,U))=g_v(\nabla^v_e Y,u)+ g_v(y,\nabla^v_e U) + 2C_v(y,u,\nabla^v_e V)$ for any $e\in T_p\M$. This proves the skew-symmetry. The last statement of $(i)$ is a direct consequence of Lemma \ref{Lemma3}.

The symmetry of $ \T^v $ for vertical vectors is because it coincides with the second fundamental form of the submersion fibers  (see for example \cite[\S 3.1]{2022HuberJavaloyes}),  while the antisymmetry $ \At^v_x v = - \At^v_v x $ under the assumption that $ v \in \hor_p $ can be obtained from the Koszul formula \eqref{koszul} under the following form, where $ W $ is a local vertical extension of $ w $, $ V $ a locally horizontal extension of $ v $ and $ X $ a locally $ g_V $-horizontal extension of $ x $,
\begin{equation}\label{vwxY}
g_v(w,\nabla^v_x V)
=
\tfrac{1}{2} g_v([X,V],w)
-C_v(x,w,\nabla^v_v V)
\,,
\end{equation}
 (recall that the Lie bracket of a projectable vector field with another vertical one is vertical and $C_v(v,\cdot,\cdot)=0$ by homogeneity). 
 By \eqref{YTA} and \eqref{vwxY},  we obtain
\begin{equation}\label{vwxy}
g_v(w,\At^v_x v)
=
g_v(w,\nabla^v_x V) 
=
\tfrac{1}{2} g_v([X,V],w)
- C_v(x,w,\nabla^v_{v} V)
\,,
\end{equation}
where,  using again \eqref{YTA}, we can compute 
\begin{multline*}
g_v([X,V],w)
=
g_v(\nabla^v_xV,w)-g_v(\nabla^v_vX,w)
\\
=
g_v(\At^v_xv,w)-g_v(\At^v_vx,w)
+2C_v(x,\nabla^v_vV,w)
\,,
\end{multline*}
such that \eqref{vwxy} simplifies to
\begin{equation*}
g_v(w,\At^v_x v)
=
\tfrac{1}{2} g_v([X,V],w)
-C_v(x,w,\nabla^v_v V)
=
\tfrac{1}{2}g_v(\At^v_x v - \At^v_v x,w)
\,.
\end{equation*}
Conclude by non-degeneracy.
\end{proof}
While we have used the tilde $ \tilde{\text{ }} $ to designate features of the base manifold $\B$ of the submersion, let us use the caret $ \hat{\text{ }} $ for the submersion fibers as submanifolds of $\M$.
\begin{definition}
Let $ \hat{Q} $ be the anisotropic tensor given for each admissible $ v $ and $ g_v $-horizontal $ x $ by $ \hat{Q}^v_x = \hat{Q}^v x = 0 $ and for each vertical $ u,w $ by
\begin{multline*}
\hat{Q}^v_u w
=
-
\left(
\T^v_v C^\sharp_v(u,w)
+
C^\sharp_v(\T^v_v u,w)
+
C^\sharp_v(u,\T^v_v w)
\right. \\ \left.
+
C^\sharp_v(C^\sharp_v(u,w)^{\top}_v,\T^v_v v)
-
C^\sharp_v(C^\sharp_v(u,\T^v_v v)^{ \top}_v,w)
-
C^\sharp_v(u,C^\sharp_v(w,\T^v_v v)^{\top}_v)
\right)^\top_v
\,.
\end{multline*}
\end{definition}
\begin{proposition}[Gauss formula] \label{GaussFormulaVert}
For any vertical $ u $, vertical admissible $ v $, and vertical $ w $ at $ p $ and any vertical extension $ W $ of $ w $,
\begin{equation*}
(\nabla^v_u W)^\top_v
=
\hat\nabla^v_u W
+
\hat{Q}^v_u w
\,.
\end{equation*}
\end{proposition}
\begin{proof}
	Completely analogous to the Gauss formula for pseudo-Finsler submanifolds (see for example \cite[Eq. (17) and Lemma 2]{2022HuberJavaloyes} for a reference using the same approach).
	\end{proof}
\begin{definition}\label{tildeQ}
Let $ \tilde{Q} $ be the anisotropic tensor given for each admissible $ v $ and vertical $ w $ by $ \tilde{Q}^v_w = \tilde{Q}^v w = 0 $ and for each $ g_v $-horizontal $ x,y $ by
\begin{equation*}
\tilde{Q}^v_x y
=
\left(\ \At^v_v C^\sharp_v(x,y) + C^\sharp_v(\At^v_v x,y) + C^\sharp_v(x,\At^v_v y) \right)^\bot_v
\,.
\end{equation*}
\end{definition}
\begin{proposition}[Dual Gauss formula]\label{VXYhor}
For each  projectable and horizontal $ V $ defined  in a neighborhood of $ p $, and projectable $ X $ and $ Y $ locally $ g_V $-horizontal,
\begin{equation*}
(\nabla^V_X Y)^\bot_V
=
(\tilde\nabla^{\tilde{V}}_{\tilde{X}} \tilde{Y})^\ast_V
+
\tilde{Q}^V_X Y.
\end{equation*}
 where  $(\cdot)^*_V$ denotes the  $g_V$-horizontal lift.
\end{proposition}
\begin{proof}
Let $ v = V_p $, $ x = X_p $ and $ y = Y_p $. Using the Koszul formula \eqref{koszul} on $ \B $
\begin{multline*}
\tilde{g}_{\tilde{v}}(\tilde{x},\tilde\nabla^{\tilde{v}}_{\tilde{y}} \tilde{Z})
=
\tilde{y}(\tilde{g}_{\tilde{V}}(\tilde{X},\tilde{Z}))
+\tilde{z}(\tilde{g}_{\tilde{V}}(\tilde{Y},\tilde{X}))
-\tilde{x}(\tilde{g}_{\tilde{V}}(\tilde{Y},\tilde{Z}))
\\
+\tilde{g}_{\tilde{v}}([\tilde{X},\tilde{Y}],\tilde{z})+\tilde{g}_{\tilde{v}}(\tilde{y},[\tilde{X},\tilde{Z}])+\tilde{g}_{\tilde{v}}(\tilde{x},[\tilde{y},\tilde{z}])
\\
-2\tilde{C}_{\tilde{v}}(\tilde{x},\tilde{z},\tilde\nabla^{\tilde{v}}_{\tilde{y}} \tilde{V})
-2\tilde{C}_{\tilde{v}}(\tilde{y},\tilde{x},\tilde\nabla^{\tilde{v}}_{\tilde{z}} \tilde{V})
+2\tilde{C}_{\tilde{v}}(\tilde{y},\tilde{z},\tilde\nabla^{\tilde{v}}_{\tilde{x}} \tilde{V})
\end{multline*}
can be lifted to
\begin{multline*}
2g_v(x,(\tilde\nabla^{\tilde{v}}_{\tilde{y}} \tilde{Z})^\ast_v)
=
y(g_V(X,Z))+z(g_V(Y,X))-x(g_V(Y,Z))
\\
+g_v([X,Y],z)+g_v(y,[X,Z])+g_v(x,[Y,Z])
\\
-2C_v(x,z,(\tilde\nabla^{\tilde{v}}_{\tilde{y}} \tilde{V})^\ast_v)
-2C_v(y,x,(\tilde\nabla^{\tilde{v}}_{\tilde{z}} \tilde{V})^\ast_v)
+2C_v(y,z,\tilde\nabla^{\tilde{v}}_{\tilde{x}} \tilde{V})
\,.
\end{multline*}
This is because by Proposition \ref{dotPartialTopProposition}, $\tilde{y}(\tilde{g}_{\tilde{V}}(\tilde{X},\tilde{Z}))=\tilde y(g_V(X,Z)\circ\sigma)=y(g_V(X,Z))$, and so on for the similar terms, and also because \[\tilde{g}_{\tilde{v}}([\tilde{X},\tilde{Y}],\tilde{z})=g_v([\tilde{X},\tilde{Y}]^*_v,z))
=g_v([X,Y]^\perp_v,z)=g_v([X,Y],z).\]
 Making use of the Koszul formula \eqref{koszul} for $\M$, it follows that
\begin{multline}\label{koszulmedia}
2g_v(x,\nabla^v_y Z - (\tilde\nabla^{\tilde{v}}_{\tilde{y}} \tilde{Z})^\ast_v)
=
-2C_v(x,z,\nabla^v_y V - (\tilde\nabla^{\tilde{v}}_{\tilde{y}} \tilde{V})^\ast_v)
-2C_v(y,x,\nabla^v_z V - (\tilde\nabla^{\tilde{v}}_{\tilde{z}} \tilde{V})^\ast_v)
\\
+2C_v(y,z,\nabla^v_x V - (\tilde\nabla^{\tilde{v}}_{\tilde{x}} \tilde{V})^*_v),
\end{multline}
where, by our assumption that $ V $ is horizontal, we may substitute $ Z $ and $ Y $ for $ V $ to furthermore obtain
$
g_v(x,\nabla^v_v V - (\tilde\nabla^{\tilde{v}}_{\tilde{v}} \tilde{V})^\ast_v)=0.$ As $(\nabla^v_v V)^\top=\At^v_vv=0$ by  part $(iii)$ of Proposition \ref{TAProposition},  we conclude that $\nabla^v_v V = (\tilde\nabla^{\tilde{v}}_{\tilde{v}} \tilde{V})^\ast_v$,
and, replacing now only $Z$ by $V$ in \eqref{koszulmedia}, we get
\begin{equation*}
g_v(x,\nabla^v_y V - (\tilde\nabla^{\tilde{v}}_{\tilde{y}} \tilde{V})^\ast_v)
=
0
\,,
\end{equation*}
which, using \eqref{YTA} under the form $ (\nabla^v_y V)^\top_v = \At^v_y v $, we find that $\nabla^v_y V - (\tilde\nabla^{\tilde{v}}_{\tilde{y}} \tilde{V})^\ast_v=\At^v_yv$, so using this  (and similar formulas replacing $y$ with $x$ and $z$)
 in \eqref{koszulmedia},  we obtain that
\begin{equation*}
g_v(x,\nabla^v_y Z - (\tilde\nabla^{\tilde{v}}_{\tilde{y}} \tilde{Z})^\ast_v)
=
-g_v(x,C^\sharp_v(z,\At^v_y v))
-g_v(x,C^\sharp_v(y,\At^v_z v))
+g_v(\At^v_x v,C^\sharp_v(y,z))
\,.
\end{equation*}
Conclude the result at $p$ by Proposition \ref{TAProposition}, namely, using that $\At^v_xv=-\At^v_vx$ and then the skew-symmetry of $\At^v_v$, and the non-degeneracy of $ g_v $. The same process completes the proof on the neighbourhood where $V$ is horizontal and $X$ and $Y$ are $g_V$-horizontal.
\end{proof}
\begin{proposition}\label{TAProposition2}
For each horizontal $ v $ at $p$ with a  local horizontal projectable extension $V$  with its projection $\tilde V$ satisfying that  $\tilde{\nabla}^{\tilde{v}}\tilde{V} = 0$, and $ g_v $-horizontal vectors $ x,y $ with  $g_V$-horizontal extensions $ X $ and $ Y $,
\begin{align}\label{Axy}
\At^v_x y
=&
\tfrac{1}{2}[X,Y]^\top_v
+
(\T+\At)^v_{C^\sharp_v(x,y)} v
+
C^\sharp_v(\At^v_x v,y)^\top_v
-
C^\sharp_v(x,\At^v_y v)^\top_v
\,,\\ \label{Axv}
\At^v_x v
=&
\tfrac{1}{2}[X,V]^\top_v\,,
\end{align}
and $ \At^v $ satisfies the almost-antisymmetry $\At^v_x y + \At^v_y x = 2(\T+\At)^v_{C^\sharp_v(x,y)} v$. 
\end{proposition}
\begin{proof}
From the Koszul formula \eqref{koszul} where $ W $ is a local vertical extension of a vertical $ w $, 
\begin{equation}\label{vwxY2}
g_v(w,\nabla^v_x Y)
=
\tfrac{1}{2} g_v([X,Y],w)
-C_v(w,y,\nabla^v_x V)
-C_v(x,w,\nabla^v_y V)
+C_v(x,y,\nabla^v_w V)
\,,
\end{equation}
and using \eqref{Lieconversion} with $Y=V$ and applying \eqref{UTA} and \eqref{YTA}, it follows that
\begin{equation*}
\nabla^v_w V
=
(\nabla^v_w V)^\top_v
+
(\nabla^v_v W)^\bot_v
=
\T^v_w v
+
\At^v_v w
\,.
\end{equation*}
Replacing the above identity in the last term  of \eqref{vwxY2},  we obtain two terms, being the first one
\begin{equation*}
C_v(x,y,\T^v_w v)
=
g_v(C^\sharp_v(x,y)^\top_v,\T^v_w v)
=
-g_v(\T^v_{C^\sharp_v(x,y)^\top_v} w,v)
=
g_v(w,\T^v_{C^\sharp_v(x,y)^\top_v} v),
\end{equation*}
where we have applied parts $(i)$ and $(ii)$ of Proposition \ref{TAProposition}, and the second one,
\begin{equation*}
C_v(x,y,\At^v_v w)
=
g_v(C^\sharp_v(x,y),\At^v_v w)
=
-g_v(\At^v_v C^\sharp_v(x,y),w)
\,,
\end{equation*}
where we have applied again part $(i)$ of Proposition \ref{TAProposition}.  Using  \eqref{YTA}  and  \eqref{vwxY2},
\begin{align}\nonumber
g_v(w,\At^v_x y)
&=
g_v(w,\nabla^v_x Y) + 2C_v(y,\nabla^v_x V,w)\\\nonumber
&= \tfrac{1}{2} g_v([X,Y],w)
+C_v(w,y,\nabla^v_x V)
-C_v(x,w,\nabla^v_y V)
+C_v(x,y,\nabla^v_w V)\\
&=\tfrac{1}{2} g_v([X,Y],w)
+C_v(w,y,\nabla^v_x V)
-C_v(x,w,\nabla^v_y V)
+g_v(w,\T^v_{C^\sharp_v(x,y)^\top_v} v)\nonumber \\
&\quad\quad\quad\quad-g_v(\At^v_v C^\sharp_v(x,y),w).\label{vwxy2}
\end{align}
Conclude \eqref{Axy} taking into account that
\begin{equation*}
\At^v_v C^\sharp_v(x,y)^\bot_v = -\At^v_{C^\sharp_v(x,y)} v
\,,
\end{equation*}
and using $\nabla^v_x V = \At^v_x v$ and $\nabla^v_y V = \At^v_y v$ by Proposition \ref{VXYhor}.  
 Finally,  \eqref{Axv} and the almost-antisymmetry of $\At$ are straightforward consequences of \eqref{Axy}. 
\end{proof}
The dual Gauss formula can be extended from Proposition \ref{VXYhor} to the following formula \eqref{GaussFormula}.
\begin{corollary}
For each horizontal $ v $, arbitrary vector $ e $ at $ p $ and arbitrary projectable vector field $ H $ with $ h = H_p $,
\begin{equation}\label{GaussFormula}
(\nabla^v_e H)^\bot_v
=
(\tilde\nabla^{\tilde{v}}_{\tilde{e}} \tilde{H})^\ast_v
+
\tilde{Q}^v_e h
+
(\T^v_e h + \At^v_e h + \At^v_h e)^\bot_v
\,.
\end{equation}
\end{corollary}
\begin{proof}
Consider a locally projectable horizontal extension $ V $ of $ v $. 
Now by definition
\begin{equation}\label{nablaHfirst}
(\nabla^v_e H)^\perp_v=(\nabla^v_e (H^\perp))^\perp_v+(\nabla^v_e (H^\top))^\perp_v
=(\nabla^v_e (H^\perp))^\perp_v+\T^v_e (h^\top_v)+\At^v_e (h^\top_v).
\end{equation}
Consider a $g_V$-horizontal projectable extension $X$ of $e^\perp_v$ and a vertical extension $U$ of $e^\top_v$. Then,
using \eqref{vertder}  to cancel the fiber derivative in the first equality,  Proposition \ref{VXYhor}, the fact that the Lie bracket of projectable vectors is vertical when one of them is vertical and \eqref{UTA},
\begin{align*}
(\nabla^v_e (H^\perp))^\perp_v&=(\nabla^v_e (H^\perp_V))^\perp_v
=(\nabla^v_X (H^\perp_V))^\perp_v+(\nabla^v_U (H^\perp_V))^\perp_v\\
&=
(\tilde\nabla^{\tilde v}_{\tilde e}\tilde H)^*_v +\tilde Q^v_e h+(\nabla^v_{H^\perp_V}U)^\perp_v
=(\tilde\nabla^{\tilde v}_{\tilde e}\tilde H)^*_v +\tilde Q^v_e h+\At^v_h (e^\top_v).
\end{align*}
Replacing the last identity in \eqref{nablaHfirst}, and taking into account that 
\[\T^v_e (h^\top_v)+\At^v_e (h^\top_v)+\At^v_h (e^\top_v)=(\T^v_e h + \At^v_e h + \At^v_h e)^\bot_v\]
as a consequence of part $(i)$ of Proposition \ref{TAProposition}, we obtain \eqref{GaussFormula}.
\end{proof}
We can thus remove the choice of an extension of $ v $ from the hypotheses.
\begin{proposition}[Dual Gauss formula]\label{GaussProposition}
For each horizontal $ v $, $ g_v $-horizontal $ x $ and $ y $ at $ p $ and any projectable extension $ Y $ of $ y $,
\begin{equation*}
(\nabla^v_x Y)^\bot_v
=
(\tilde\nabla^{\tilde{v}}_{\tilde{x}} \tilde{Y})^\ast_v
+
\tilde{Q}^v_x y
\,.
\end{equation*}
\end{proposition}
%
One very nicely behaved consequence of the dual Gauss formula is the relationship between geodesics of the base and horizontal geodesics.
\begin{corollary}\label{horizontalGeodesicCorollary}
A horizontal curve is a geodesic of $(\M,L)$ if and only if it is the horizontal lift of a geodesic of $(\B,\tilde L)$.
\end{corollary}
\begin{proof}
By Proposition \ref{GaussProposition}, the projection of a horizontal curve $ \gamma $ satisfying  $ D^{\dot\gamma}_\gamma \dot\gamma = 0 $  satisfies the geodesic equation (consider for example a horizontal projectable vector field $V$ that locally extends $\dot\gamma$). Conversely, the same proposition allows us to deduce $ (D^{\dot\gamma}_{\gamma} \dot\gamma)^\bot_{\dot\gamma} = 0 $ for the horizontal lift of a geodesic, while by Lemma \ref{Lemma3} 
\begin{equation*}
(D^{\dot\gamma}_{\gamma} \dot\gamma)^\top_{\dot\gamma}
=
\At^{\dot\gamma}_{\dot\gamma} \dot\gamma
-
2C^\sharp_{\dot\gamma}(\dot\gamma,D^{\dot\gamma}_{\gamma} \dot\gamma)^\top_{\dot\gamma}
\,,
\end{equation*}
where the last term is zero by homogeneity, and the first term to the right is zero by part $(iii)$ of Proposition \ref{TAProposition}.
\end{proof}
\begin{corollary}\label{horizontalEverywhereCorollary}
If a geodesic of $(\M,L)$ is horizontal at one instant, then it is horizontal everywhere.
\end{corollary}
\begin{proof}
Assume that $\gamma$ is horizontal at $t_0\in [a,b]$. Then consider the lift $\beta$ of the geodesic with initial velocity 
$\dif\sigma(\dot\gamma(t_0))$  (recall Corollary \ref{horizontalGeodesicCorollary}).  As both $\gamma$ and $\beta$ have the same initial velocity, they must coincide and $\gamma$ is horizontal everywhere.
\end{proof}
 The last two corollaries were obtained in \cite[Theorems 3.1 and 5.3]{2001AlDur} for Finsler metrics  using two different methods, the minimizing property of geodesics (which holds only in the positive definite case) and symplectic reduction, which can be also applied to pseudo-Finsler submersions. 
\subsection{Covariant derivatives of T and A}\label{covderTA}

Note that skew-symmetry from Proposition \ref{TAProposition} implies that $ \nabla \T $ and $ \nabla \At $ are skew-symmetric. 
Algebraicity of $ \nabla \T $ and $ \nabla \At $ can be imported intact from \cite{1966ONeill} as follows.
\begin{lemma}\label{Lemma4}
For each admissible $ v $, vertical $ u,w $ and $ g_v $-horizontal $ x,y $ and arbitrary $ e $ at $ p $,
\begin{align*}
(\nabla_w \At)^v_ue
&=
-\At^v_{T^v_w u}e
\,, &\quad
(\nabla_x \At)^v_we
&=
-\At^v_{\At^v_x w}e
\,, \\
(\nabla_w \T)^v_ye
&=
-\T^v_{\T^v_w y}e
\,, &\quad
(\nabla_x \T)^v_ye
&=
-\T^v_{\At^v_x y}e
\,.
\end{align*}
More succinctly, for every $h$ at $p$
\begin{equation*}
(\nabla_h \At)^v_w e = -\At^v_{(\T+\At)^v_hw}e
\,, \quad
(\nabla_h \T)^v_x e = -\T^v_{(\T+\At)^v_hx}e
\,.
\end{equation*}
\end{lemma}
\begin{proof}
For a locally admissible extension $ V $ of $ v $ satisfying $ \nabla^v V = 0 $, a vertical extension $U$ of $u$ and a $g_V$-horizontal extension $Y$ of $y$, all the identities are a direct consequence of definitions taking into account that $\At^v_u=0$ for any vertical $u$ and $\T^v_x=0$ for any $g_v$-horizontal $x$ and the identities of Lemma \ref{Lemma3}.
\end{proof}
\begin{lemma}\label{Lemma5}
For each admissible $ v $, vertical $ u,w $ and $ g_v $-horizontal $ x $ and arbitrary $e$ and $h$ at $ p $,
\begin{align*}
&
((\nabla_w \At)^v_e u)^\top_v
=
(\T^v_w \At^v_e - \At^v_e \T^v_w) u
\,, \quad
((\nabla_x \At)^v_e u)^\top_v
=
(\At^v_x \At^v_e - \At^v_e \At^v_x) u
\,, \\
&
((\nabla_w \T)^v_e u)^\top_v
=
(\T^v_w \T^v_e - \T^v_e \T^v_w) u
\,, \quad
((\nabla_x \T)^v_e u)^\top_v
=
(\At^v_x \T^v_e - \T^v_e \At^v_x) u
\,,
\end{align*}
and for each $ g_v $-horizontal $ y $ at $ p $
\begin{align*}
&
((\nabla_w \At)^v_e y)^\bot_v
=
(\T^v_w \At^v_e - \At^v_e \T^v_w) y
\,, \quad
((\nabla_x \At)^v_e y)^\bot_v
=
(\At^v_x \At^v_e -\At^v_e \At^v_x) y
\,, \\
&
((\nabla_w \T)^v_e y)^\bot_v
=
(\T^v_w \T^v_e - \T^v_e \T^v_w) y
\,, \quad
((\nabla_x \T)^v_e y)^\bot_v
=
(\At^v_x \T^v_e - \T^v_e \At^v_x) y
\,.
\end{align*}
More synthetically, 
\begin{align}\label{nablaT+Au}
((\nabla_e (\T+\At))^v_h u)^\top_v
& =
(\T+\At)^v_e (\T+\At)^v_h u - (\T+\At)^v_h (\T+\At)^v_e u
\,, \\
\label{nablaT+Ay}((\nabla_e (\T+\At))^v_h y)^\bot_v
& =
(\T+\At)^v_e (\T+\At)^v_h y - (\T+\At)^v_h (\T+\At)^v_e y
\,.
\end{align}
\end{lemma}
\begin{proof}
Consider a locally admissible extension $ V $ of $ v $ satisfying $ \nabla^v V = 0 $, an arbitrary extension $ E $ of $ e $ and any vertical extension $ U $ of $ u $.  Then 
\begin{equation*}
(\nabla_h \T)^v_e u
=
\nabla^v_h (\T^V_E U)
-
\T^v_{\nabla^v_h E} u
-
\T^v_e \nabla^v_h U
\,,
\end{equation*}
whose second term will vanish when taking the vertical part, while the vertical part of the third term is by part $(i)$ of Proposition \ref{TAProposition} and \eqref{UTA}, 
\begin{equation*}
-(\T^v_e \nabla^v_h U)^\top_v=-\T^v_e (\nabla^v_h U)^\perp_v
=
-\T^v_e (\T+\At)^v_h U
\,.
\end{equation*}
The vertical part of the first term may also be computed into a similar form as follows. For any vertical extension $ W $ of $ w $, by Lemma \ref{Lemma3}, noting that the product $ g_V((\T+\At)^V_E U,W) $ is zero by properties of $ \T $ and $ \At $, and using also \eqref{UTA} and skew-symmetry of $\T$ and $\At$,
\begin{multline*}
g_v(\nabla^v_h (\T^V_E U),w)
=
-
g_v(	\T^v_e u,\nabla^v_h W)
-
2C_v(\T^v_e u,w,\nabla^v_h V)
\\
=
-g_v((\T^v_e u,(\T+\At)^v_h w)
=
g_v((\T+\At)^v_h \T^v_e u,w)
\,.
\end{multline*}
It follows then, putting together the last three identities, that 
\[((\nabla_h \T)^v_e u)^\top_v=-\T^v_e (\T+\At)^v_h u+(\T+\At)^v_h \T^v_e u.\]
With analogous computations, we obtain other three identities
\begin{align*}
((\nabla_h \At)^v_e u)^\top_v&=-\At^v_e (\T+\At)^v_h u+(\T+\At)^v_h \At^v_e u,\\
((\nabla_h \T)^v_e y)^\bot_v&=-\T^v_e (\T+\At)^v_h y+(\T+\At)^v_h \T^v_e y,\\
((\nabla_h \At)^v_e y)^\bot_v&=-\At^v_e (\T+\At)^v_h y+(\T+\At)^v_h \At^v_e y.
\end{align*}
All the identities in Lemma \ref{Lemma5} follow from the above four, by taking into account that $\T_x=\At_w=0$.
\end{proof}

\section{Fundamental equations}\label{fundeq}

Our aim is to generalise the fundamental equations $ \lbrace 0 \rbrace $ to $ \lbrace 4 \rbrace $ in \cite{1966ONeill}. To this end, let us compute the curvature tensor in terms of the following curvature tensors.
\begin{definition}\label{RTopBotDefinition} 
Let $ R^\top $ and $ R^\bot $ be given for each admissible $ v $  and tangent vectors $ b,e,h $ at $p$ with extensions $ B,E,H $ by
\begin{equation*}
R^\top_v(b,e)w
=
\left(
\nabla^v_b (\nabla_E H^\top)^\top
-
\nabla^v_e (\nabla_B H^\top)^\top
-
\nabla^v_{[B,E]} H^\top
\right)^\top_v
\,,
\end{equation*}
and replacing the vertical operator $ \text{ }^\top $ by the horizontal operator $ \text{ }^\bot $
\begin{equation*}
R^\bot_v(b,e)h
=
\left(
\nabla^v_b (\nabla_E H^\bot)^\bot
-
\nabla^v_e (\nabla_B H^\bot)^\bot
-
\nabla^v_{[B,E]} H^\bot
\right)^\bot_v
\,.
\end{equation*}
\end{definition}
One can easily checked that $R^\top$ and $R^\bot$ are anisotropic tensors by the same method as the proof that the curvature tensor of a Riemannian manifold is indeed a tensor.
\begin{proposition}\label{RTopRBotProposition}
With the same notation, for $v$ admissible with a locally admissible extension $V$ and $w$ vertical with vertical extension $W$,
\begin{multline*}
R^\top_v(b,e) w
=
\left(
\nabla^v_b (\nabla^V_E W)^\top_V
-
\nabla^v_e (\nabla^V_B W)^\top_V
-
\nabla^v_{[B,E]} W
\right)^\top_v
\\
-
2C^\sharp_v(\T^v_ew,\nabla^v_bV)^\top_v
+
2C^\sharp_v(\T^v_bw,\nabla^v_eV)^\top_v
\\
-
P_v(e,w,\nabla^v_bV)^\top_v
+
P_v(b,w,\nabla^v_eV)^\top_v
\,,
\end{multline*}
and
\begin{multline*}
R^\bot_v(b,e)h
=
\left(
\nabla^v_b (\nabla^V_E H^\bot_V)^\bot_V
-
\nabla^v_e (\nabla^V_B H^\bot_V)^\bot_V
-
\nabla^v_{[B,E]} H^\bot_V
\right)^\bot_v
\\
-
P_v(e,h^\bot_v,\nabla^v_bV)^\bot_v
+
P_v(b,h^\bot_v,\nabla^v_eV)^\bot_v
\\
+
2 (\T^v_e+\At^v_e) C^\sharp_v(h^\bot_v,\nabla^v_bV)^\top_v
-
2 (\T^v_b+\At^v_b) C^\sharp_v(h^\bot_v,\nabla^v_eV)^\top_v
\,.
\end{multline*}
\end{proposition}
\begin{proof}
In the first case, 
 as $W^\top=W$, we only have to apply the chain rule  and  \eqref{vertder} to compute, 
\begin{align*}
\nabla^v_b (\nabla_E W)^\top
=&
\nabla^v_b (\nabla^V_E W)^\top_V
-
(\dot\partial (\nabla_e W)^\top)_v(\nabla^v_bV)
\\=&
\nabla^v_b (\nabla^V_E W)^\top_V
-2C^\sharp_v(\T^v_ew,\nabla^v_bV)^\top_v
-
P_v(e,w,\nabla^v_bV)^\top_v
\,.
\end{align*}
 Putting together the above identity and an analogous one commuting the role of $e$ and $b$, one concludes the expression for $R^\top$. 

In the second case, 
by \eqref{vertder} 
\begin{equation}\label{nablabrac}
\nabla^v_{[B,E]} H^\bot
=
\nabla^v_{[B,E]} H^\bot_V
-
(\dot\partial (H^\bot))_v(\nabla^v_{[B,E]} V)
=
\nabla^v_{[B,E]} H^\bot_V
+
2C^\sharp_v(h^\bot_v,\nabla^v_{[B,E]}V)^\top_v
\,,
\end{equation}
whose last term will vanish under $(\cdot)^\bot_v$, while 
\begin{equation*}
(\nabla^v_b (\nabla_E H^\bot)^\bot)^\bot_v
=
(\nabla^v_b (\nabla^V_E H^\bot)^\bot_V)^\bot_v
-
((\dot\partial (\nabla_e H^\bot)^\bot)_v(\nabla^v_bV))^\bot_v
\,.
\end{equation*}
Using the chain rule and \eqref{dotPartialNablaMathcalX},
\begin{align*}
((\dot\partial (\nabla_e H^\bot)^\bot)_v(\nabla^v_bV))^\bot_v
=&
((\dot\partial (\nabla^v_e H^\bot)^\bot)_v(\nabla^v_bV))^\bot_v
+
(((\dot\partial (\nabla_e H^\bot))_v(\nabla^v_bV))^\bot_v)^\bot_v
\\
=&
((\dot\partial (\nabla_e H^\bot))_v(\nabla^v_bV))^\bot_v
\\=&
P_v(e,h^\bot_v,\nabla^v_bV)^\bot_v
+
((\nabla_e(\dot\partial H^\bot))_v(\nabla^v_bV))^\bot_v
\,,
\end{align*}
because by \eqref{vertder}, the first term to the right and the term coming from the third one in \eqref{dotPartialNablaMathcalX} are the $g_v$-horizontal part of a vertical vector. 
Lastly, again by \eqref{vertder},
\begin{multline*}
(\nabla^V_E H^\bot)^\bot_V
=
(\nabla^V_E H^\bot_V)^\bot_V
-
((\dot\partial (H^\bot))_V(\nabla^V_EV))^\bot_V
\\
=
(\nabla^V_E H^\bot_V)^\bot_V
+
2(C^\sharp_V(H^\bot_V,\nabla^V_EV)^\top_V)^\bot_V
=
(\nabla^V_E H^\bot_V)^\bot_V
\,.
\end{multline*}
All that remains is to compute $((\nabla_e(\dot\partial (H^\bot)))_v(\nabla^v_bV))^\bot_v$, which we may obtain from \eqref{vertder} under the form
\begin{equation*}
(\nabla_e (\dot\partial (H^\bot))_v(\nabla^v_bV)
=
-2(\nabla_e (C^\sharp(H^\bot,\cdot)^\top))_v(\nabla^v_bV)
\end{equation*}
and by extending the definition of the tensors $\T$ and $\At$ to anisotropic vector fields, we obtain 
\begin{align*}
-2((\nabla_e (C^\sharp(H^\bot,\cdot)^\top))_v(\nabla^v_bV))^\bot_v
=&
-2(\nabla^v_eC^\sharp(H^\bot,\nabla_BV)^\top)^\bot_v
\\=&
-2(\T^v_e + \At^v_e)C^\sharp_v(h^\bot_v,\nabla^v_bV)^\top_v
\,.
\end{align*}
 Putting together all the last identities, we conclude that 
\begin{multline*}
(\nabla^v_b (\nabla_E H^\bot)^\bot)^\bot_v=(\nabla^v_b (\nabla^V_E H^\bot_V)^\bot_V)^\bot_v-
P_v(e,h^\bot_v,\nabla^v_bV)^\bot_v
\\
+2(\T^v_e + \At^v_e)C^\sharp_v(h^\bot_v,\nabla^v_bV)^\top_v
\,.
\end{multline*} 
Taking into account the expression obtained by commuting $e$ and $b$ and \eqref{nablabrac}, one deduces the expression for $R^\bot$.
\end{proof}
\begin{proposition}\label{RTopProposition}
For each admissible $ v $, arbitrary $ e $ and $ h $ at $ p $, vertical $w$ and $ g_v $-horizontal $ x $,
\begin{equation*}
R^\top_v(e,h)x
=
-
2 C^\sharp_v(R_v(e,h)v,x)^\top_v
\,,
\quad
R^\bot_v(e,h)w
=
0
\,.
\end{equation*}
\end{proposition}
\begin{proof}
The second identity is merely a consequence of the fact that for any vertical extension $W$ of $w$, the anisotropic vector field $W^\top$ is constant and identical to $W$ on each tangent space while $W^\bot$ is identically equal to zero. As for the first identity,
consider a locally admissible extension $ V $ of $ v $ satisfying that $ \nabla^v V = 0 $, extensions $ E $ of $e$ and $ H $ of $ h $ whose Lie bracket $[E,H]$ vanishes at $ p $ and a local $ g_V $-horizontal extension $ X $ of $ x $.
Note how by Proposition \ref{dotPartialTopProposition}
\begin{equation*}
(\nabla^V_H X^\top)^\top_V
=
(\nabla^V_H X^\top_V)^\top_V
-
((\dot\partial X^\top)(\nabla^V_H V))^\top_V
=
-2C^\sharp_V(X,\nabla^V_H V)^\top_V
\,.
\end{equation*}
Differentiating $ \Xi_{v}(\cdot,\cdot) = -2(C^\sharp_v(\cdot,\cdot))^\top_v $, with $\nabla^v_HV=\nabla^v_eV=0$ for our choice of $V$,
\begin{multline*}
\nabla^v_e ((\nabla^V_H X^\top)^\top_V)
=
\nabla^v_e (\Xi_V(X,\nabla^V_H V))
=
\Xi_v(x, \nabla^v_e \nabla^V_H V)	
=
-2C^\sharp_v(x,\nabla^v_e \nabla^V_H V)^\top_v
\,.
\end{multline*}
Reinserting into
\begin{equation*}
(\nabla^v_e (\nabla_H X^\top)^\top)^\top_v
=
(\nabla^v_e (\nabla^V_H X^\top)^\top_V)^\top_v
=
-2 C^\sharp_v(x,\nabla^v_e \nabla^V_H V)^\top_v
=
-2 C^\sharp_v(x,\nabla^v_e \nabla_H V)^\top_v
\end{equation*}
concludes the first identity up to commuting $ e $ and $ h $.
\end{proof}
Observe how, if $ v $ is vertical, then we may relate $ R^\top_v $ with the intrinsic curvature tensor $ \hat{R} $ of the submersion fibers. In analogy to the vertical derivative $P$ of the Chern connection $\nabla$ of the ambient pseudo-Finsler manifold $(\M,L)$, let $\hat P$ designate that of $\hat\nabla$ and $\tilde P$ that of $\tilde\nabla$.
\begin{proposition}\label{RTopVProposition}
For each vertical admissible $ v $ and vertical $ s $, $ u $ and $ w $ at $ p $,
 \begin{multline*}
 R^\top_v(u,w)s
=
\hat R_v(u,w)s
+
(\hat\nabla_u\hat Q)^v_ws - (\hat\nabla_w\hat Q)^v_us
+
\hat Q^v_u\hat Q^v_ws - \hat Q^v_w\hat Q^v_us
\\
-
P_v(w,s,\hat Q^v_uv+\T^v_uv)^\top_v
+
P_v(u,s,\hat Q^v_wv+\T^v_wv)^\top_v
\\
-
2C^\sharp_v(\mathbf T^v_ws,\hat Q^v_uv+\mathbf T^v_uv)^\top_v
+
2C^\sharp_v(\mathbf T^v_us,\hat Q^v_wv+\mathbf T^v_wv)^\top_v
\,. 
\end{multline*} 
\end{proposition}
\begin{proof}
 It follows from  \cite[Theorem 5]{2022HuberJavaloyes}, taking into account that 
we may differentiate the Gauss formula  in Proposition \ref{GaussFormulaVert}  with respect to the  fiber derivative in the direction of $\hat Q^v_wv$   to obtain
\begin{equation*}
P_v(u,s,\hat Q^v_wv)^\top_v
+
2C^\sharp_v(\mathbf T^v_us,\hat Q^v_wv)^\top_v
=
\hat P_v(u,s,\hat Q^v_wv)
+
 (\dot\partial\hat Q)^v_us(\hat Q^v_wv) 
\,.
\end{equation*}
 Here we have used the chain rule and \eqref{vertder}. 	
\end{proof}
Similarly, if $v$ is horizontal, then we may relate $R^\bot_v$ with the horizontal lift $\tilde{R}^\ast_v$ of the curvature tensor of the base manifold.
\begin{proposition}\label{RBotVProposition}
For each horizontal $ v $ at $ p $ and $ g_v $-horizontal $ x $, $ y $ and $ z $,
\begin{align*}
& R^\bot_v(x,y)z
\\
& \quad =
\tilde{R}^\ast_v(\tilde{x},\tilde{y})\tilde{z}
+
\At^v_z\At^v_y x
-
\At^v_z\At^v_x y
+
((\nabla_x \tilde{Q})^v_y z - (\nabla_y \tilde{Q})^v_x z)^\bot_v
+
\tilde{Q}^v_y \tilde{Q}^v_x z - \tilde{Q}^v_x \tilde{Q}^v_y z
\\ & \qquad
+
2 \At^v_z C^\sharp_v(y,\At^v_xv)^\top_v
+
2 \At^v_y C^\sharp_v(z,\At^v_xv)^\top_v
-
2 \At^v_z C^\sharp_v(x,\At^v_yv)^\top_v
-
2 \At^v_x C^\sharp_v(z,\At^v_yv)^\top_v
\\ & \qquad
-
P_v(y,z,\At^v_xv)^\bot_v
+
P_v(x,z,\At^v_yv)^\bot_v
+
((\dot\partial\tilde Q)^v_yz(\At^v_xv))^\bot_v
-
((\dot\partial\tilde Q)^v_xz(\At^v_yv))^\bot_v
\,.
\end{align*}
\end{proposition}
\begin{proof}
Consider an extension $ \tilde{V} $ of $ \tilde{v} $ satisfying $ \tilde\nabla^{\tilde{v}} \tilde{V} = 0 $, and let $ V $ be, locally, the horizontal lift of $ \tilde{V} $. In particular, by the dual Gauss formula $ (\nabla^v V)^\bot_v $ vanishes on $g_v$-horizontal vectors, while, by Lemma \ref{Lemma3}, $\nabla^v_x V = \At^v_x v$ and $\nabla^v_y V = \At^v_y v$. Let $ \tilde{X} $, $ \tilde{Y} $ and $ \tilde{Z} $ be extensions of $ \tilde{x} $, $ \tilde{y} $ and $ \tilde{z} $ with mutual Lie brackets that vanish at $ p $. Let $ X $, $ Y $ and $ Z $ be, locally, their respective $ g_V $-horizontal lifts. In particular, $ [X,Y]^\bot_v = [Y,Z]^\bot_v = [Z,X]^\bot_v = 0 $.
From Proposition \ref{RTopRBotProposition}, it suffices to compute $(\nabla^v_{[X,Y]}  Z)^\bot_v$  (recall that $Z^\bot_V=Z$)  and by Proposition \ref{GaussProposition}
\begin{equation*}
(\nabla^v_x(\nabla^V_YZ)^\bot_V)^\bot_v
=
(\nabla^v_x (\tilde\nabla^{\tilde V}_{\tilde y} \tilde Z)^\ast_V)^\bot_v
+
(\nabla^v_x \tilde Q^V_Y Z)^\bot_v
\,.
\end{equation*}
Observe that we may use the almost-antisymmetry of $\At$ to rewrite \eqref{Axy} under the form
\begin{equation}\label{AtvxyLieAtvyx}
\tfrac{1}{2}
\At^v_x y
=
\tfrac{1}{2}[X,Y]^\top_v
+
\tfrac{1}{2}
\At^v_y x
+
C^\sharp_v(\At^v_x v,y)^\top_v
-
C^\sharp_v(x,\At^v_y v)^\top_v
\,.
\end{equation}
Then, using the above identity, \eqref{Lieconversion} (recall that $[X,Y]$ is vertical by assumption) 
and part $(iii)$ of Proposition \ref{TAProposition}, we get
\begin{equation*}
(\nabla^v_{[X,Y]} Z)^\bot_v
=
\At^v_z [X,Y]
=
\At^v_z \At^v_x y
-
\At^v_z \At^v_y x
+
2 \At^v_z C^\sharp(\At^v_v x,y)^\top_v
-
2 \At^v_z C^\sharp_v(x,\At^v_v y)^\top_v
\,.
\end{equation*}
By the dual Gauss formula, and since for our choice of extension $\tilde\nabla^{\tilde v}\tilde V=0$,
\begin{multline*}
(\nabla^v_x (\tilde\nabla^{\tilde V}_{\tilde y} \tilde Z)^\ast_V)^\bot_v
+
(\nabla^v_x \tilde Q^V_Y Z)^\bot_v
=
(\tilde\nabla^{\tilde v}_{\tilde x} \tilde\nabla_{\tilde Y} \tilde Z)^\ast_v
+
\tilde Q^v_x(\tilde\nabla^{\tilde v}_{\tilde y} \tilde Z)^\ast_v
\\
+
((\nabla_x\tilde Q)^v_yz)^\bot_v
+
\tilde Q^v_{\nabla^v_xY}z
+
\tilde Q^v_y(\tilde\nabla^{\tilde v}_{\tilde x}\tilde Z)^\ast_v
+
\tilde Q^v_y\tilde Q^v_xz
+
((\dot\partial\tilde Q)^v_yz(\At^v_xv))^\bot_v
\,.
\end{multline*}
When subtracting to the formula we have obtained the formula commuted in $x$ and $y$, the pairs of $\tilde Q^v(\tilde\nabla^{\tilde v}\tilde Z)^\ast_v$ terms cancel each other out, as well as the pair of $\tilde Q^v_{\nabla^v_xY}z$ and $\tilde Q^v_{\nabla^v_yX}z$ by our assumption that $[X,Y]^\bot_v=0$.
\end{proof}
\begin{theorem}[Unified fundamental equation of a pseudo-Finsler submersion]
For each admissible $ v $ and arbitrary vectors $ e $ and $ h $ at $ p $,
\begin{multline}\label{Unified}
R_v(e,h)
=
R^\top_v(e,h)
+
R^\bot_v(e,h)
+
(\nabla_e (\T+\At))^v_h - (\nabla_h (\T+\At))^v_e
\\
+
(\T+\At)^v_h(\T+\At)^v_e - (\T+\At)^v_e(\T+\At)^v_h
\,,
\end{multline}
where $R$ stands for the curvature tensor of the Chern connection of the ambient manifold, $R^\top$ and $R^\bot$ are as defined in Definition \ref{RTopBotDefinition} and $\T,\At$ are the O'Neill tensors of Definition \ref{ATDefinition}.
\end{theorem}
\begin{proof}
For a choice of extensions whose Lie brackets vanish at $ p $,
\begin{equation}
R_v(e,h)
=
\nabla^v_e \nabla_H
-
\nabla^v_h \nabla_E
\,.
\end{equation}
Consider any vector field $ B $ with $ b = B_p $. By definition of the tensors $ \T $ and $ \At $,
\begin{equation*}
\nabla^v_e \nabla_H B
=
\nabla^v_e (\nabla_H B^\top)^\top
+
\nabla^v_e ((\T+\At)_H B)
+
\nabla^v_e (\nabla_H B^\bot)^\bot
\,,
\end{equation*}
where we may further expand the first and last terms as
\begin{align*}
\nabla^v_e (\nabla_H B^\top)^\top
&=
(\nabla^v_e (\nabla_H B^\top)^\top)^\top_v
+
(\T+\At)^v_e (\nabla^v_h B^\top)^\top_v,\\
\nabla^v_e (\nabla_H B^\perp)^\perp
&=
(\nabla^v_e (\nabla_H B^\perp)^\perp)^\perp_v
+
(\T+\At)^v_e (\nabla^v_h B^\perp)^\perp_v,
\end{align*}
and the middle term as
\begin{align*}
\nabla^v_e ((\T+\At)_H B)
=&(\nabla_e (\T+\At))^v_h b
+(\T+\At)^v_{\nabla^v_e H} b+(\T+\At)^v_{h}{\nabla^v_e B} \\
=&(\nabla_e (\T+\At))^v_h b
+
(\T+\At)^v_{\nabla^v_e H} b
\\
&+
(\T+\At)^v_h (\nabla^v_e B^\top)^\top_v
+
(\T+\At)^v_h (\T+\At)^v_e b
+
(\T+\At)^v_h (\nabla^v_e B^\bot)^\bot_v
\end{align*}
to obtain a long expression, to which we may substract the corresponding expressions commuted in $e$ and $h$. Pairs of $(\T+\At)^v(\nabla^vB^\top)^\top_v$ and $(\T+\At)^v(\nabla^vB^\bot)^\bot_v$ terms cancel each other out. Cancel $(\T+\At)^v_{\nabla^v_eH}b$ with $-(\T+\At)^v_{\nabla^v_hE}b$ by our assumption that $[E,H]_p=0$ to conclude.
\end{proof}
\begin{corollary}[Fundamental equations of a pseudo-Finsler submersion]\label{fundeqgeneral}
For each admissible $v$, vertical $s',s,u,w$ and $g_v$-horizontal $x,y,z,z'$
\begin{align*}\tag*{\{0\}}\label{0}
&
g_v(R_v(w,u)s,s')
=
g_v(R^\top_v(w,u)s,s')
 +
g_v(\T^v_ws,\T^v_us')
-
g_v(\T^v_us,\T^v_ws')
\,,
\\ & \tag*{\{1\}}\label{1}
g_v(R_v(w,u)s,z)
=
g_v((\nabla_w \T)^v_us,z) - g_v((\nabla_u \T)^v_ws,z)
\,,
\\ & \tag*{\{1'\}}\label{1p}
g_v(R_v(x,u)s,w)
=
g_v(R^\top_v(x,u)s,w)
+
g_v(\T^v_us,\At^v_xw)
-
g_v(\At^v_xs,\T^v_uw)
\,,
\\ & \tag*{\{2\}}\label{2}
g_v(R_v(x,u)s,z)
=
g_v((\nabla_x \T)^v_us,z)
-
g_v((\nabla_u \At)^v_xs,z)
-
g_v(\At^v_{\At^v_xu}s,z)
-
g_v(\T^v_ux,\T^v_sz)
\,,
\\ & \tag*{\{2'\}}\label{2p}
g_v(R_v(x,y)s,w)
=
g_v(R^\top_v(x,y)s,w)
+
g_v(\At^v_ys,\At^v_xw)
-
g_v(\At^v_xs,\At^v_yw)
\,,
\\ & \tag*{\{3\}}\label{3}
g_v(R_v(x,y)s,z)
=
g_v((\nabla_x \At)^v_ys,z) - g_v((\nabla_y \At)^v_xs,z)
+
g_v(\At^v_yx,\T^v_sz)
-
g_v(\At^v_xy,\T^v_sz)
\,,
\\ & \tag*{\{4\}}\label{4}
g_v(R_v(x,y)z,z')
=
g_v(R^\bot_v(x,y)z,z')
 +
g_v(\At^v_xz,\At^v_yz')
-
g_v(\At^v_yz,\At^v_xz') 
\,.
\end{align*}
\end{corollary}
\begin{proof}
Recall that, by Proposition \ref{RTopProposition}, $ R^\bot_v(w,u)s = R^\bot_v(x,u)s = R^\bot_v(x,y)s = 0 $.
 Moreover, using \eqref{nablaT+Au} and being $e,h$ arbitrary vectors at $p$,
\begin{multline}\label{NaTA}
(\nabla_e (\T+\At))^v_hs - (\nabla_h (\T+\At))^v_es
\\
+
(\T+\At)^v_h(\T+\At)^v_es - (\T+\At)^v_e(\T+\At)^v_hs
\\
=((\nabla_e (\T+\At))^v_hs - (\nabla_h (\T+\At))^v_es)^\bot_v
\\
+
(\T+\At)^v_e(\T+\At)^v_hs - (\T+\At)^v_h(\T+\At)^v_es\,. 
\end{multline}
Then the six first identities are obtained from \eqref{Unified} using the last identity together with the observation about $R^\bot$,  an application of Lemma \ref{Lemma4} and the properties of $\T$ and $\At$ in Proposition \ref{TAProposition}. In particular,
\begin{equation*}
 ((\nabla_w \At)^v_u s - (\nabla_u \At)^v_w s)^\bot_v=- \At^v_{\T^v_wu} s + \At^v_{\T^v_uw} s = 0 
\end{equation*}
by Lemma \ref{Lemma4} and the symmetry of $\T$ to get $\{0\}$ and $\{1\}$,
\[
((\nabla_x \At)^v_u s - (\nabla_u \T)^v_x s)^\bot_v
=-
\At^v_{\At^v_xu} s
+
\T^v_s\T^v_ux
\]
where we have used $\T^v_{\T^v_ux}s=\T^v_s\T^v_ux$, to get $\{1'\}$ and $\{2\}$,
\[ ((\nabla_x \T)^v_y s - (\nabla_y \T)^v_x s)^\bot_v=
-
\T^v_s\At^v_xy
+
\T^v_s\At^v_yx
\]
where we have used $ \T^v_{\At^v_xy} s = \T^v_s \At^v_xy  $, to get $\{2'\}$ and $\{3\}$. 
For the remaining identity $\{4\}$, we proceed analogously. Using \eqref{nablaT+Ay}
\begin{multline*}
(\nabla_e (\T+\At))^v_hz - (\nabla_h (\T+\At))^v_ez
\\
+
(\T+\At)^v_h(\T+\At)^v_ez - (\T+\At)^v_e(\T+\At)^v_hz
\\
=((\nabla_e (\T+\At))^v_hz - (\nabla_h (\T+\At))^v_ez)^\top_v
\\
+
(\T+\At)^v_e(\T+\At)^v_hz - (\T+\At)^v_h(\T+\At)^v_ez
\end{multline*} 
and by Lemma \ref{Lemma4},
\[ ((\nabla_x \T)^v_y z - (\nabla_y \T)^v_x z)^\top_v
=-
\T^v_{\At^v_xy} z
+
\T^v_{\At^v_yx} z\,.\]
\end{proof}
 Let us recall one of the most important geometric invariants of a pseudo-Finsler manifold and let us introduce the analogous definitions for the vertical and horizontal parts of the curvature tensor presented in Definition \ref{RTopBotDefinition}. 
\begin{definition}
For a flagpole $v\in\Ad$ and flag $e\in T_{\pi(v)}\M$ such that $L(v)g_v(e,e)\neq g_v(v,e)^2$, denote by $K_v(e)$ the flag curvature of $(\M,L)$ with the following definition:
\begin{equation*}
K_v(e)=\frac{g_v(R_v(v,e)e,v)}{L(v)g_v(e,e)-g_v(v,e)^2}
\,,
\end{equation*}
 while the {\em vertical flag curvature} $K_v^\top(e)$ and the {\em  horizontal flag curvature} $K_v^\bot(e)$ are defined in the same way replacing $R$ by $R^\top$ and $R^\bot$, respectively. 
\end{definition}
\begin{corollary}\label{flagcurgen}
For each vertical $w$, admissible $v$ and $g_v$-horizontal $x$,
\begin{multline*}
K_v(w)
=
 K^\top_v(w)+\frac{g_v((\nabla_v\T)^v_ww-(\nabla_w(\T+\At))^v_vw,v)}{L(v)g_v(w,w)-g_v(v,w)^2}
\\
+
\frac{-g_v(\At^v_{(\T+\At)^v_vw}w,v)+g_v(\T^v_ww,(\T+\At)^v_vv)-g_v((\T+\At)^v_vw,\T^v_wv)}{L(v)g_v(w,w)-g_v(v,w)^2}
\,,
\end{multline*}
and
\begin{multline*}
K_v(x)
=  K^\top_v(x)+K^\bot_v(x)+
\frac{g_v((\nabla_v\At)^v_xx-(\nabla_x(\T+\At))^v_vx,v)}{L(v)g_v(x,x)-g_v(v,x)^2}
\\
+
\frac{-g_v(\T^v_{(\T+\At)^v_vx}x,v)+g_v(\At^v_xx,(\T+\At)^v_vv)-g_v((\T+\At)^v_vx,\At^v_xv)}{L(v)g_v(x,x)-g_v(v,x)^2}
\,.
\end{multline*}
\end{corollary}
\begin{proof}
By \eqref{Unified}, 
recalling that $R^\bot_v(\cdot,\cdot)w=0$
\begin{multline*}
g_v(R_v(v,w)w,v)
=
g_v(R^\top_v(v,w)w+(\nabla_v(\T+\At))^v_ww-(\nabla_w(\T+\At))^v_vw,v)
\\
+g_v(\T^v_ww,(\T+\At)^v_vv)
-
g_v((\T+\At)^v_vw,\T^v_wv)
\,,
\end{multline*}
and similarly
\begin{multline*}
g_v(R_v(v,x)x,v)
=
g_v((R^\top+R^\bot)_v(v,x)x+(\nabla_v(\T+\At))^v_xx-(\nabla_x(\T+\At))^v_vx,v)
\\
+
g_v(\At^v_xx,(\T+\At)^v_vv)
-
g_v((\T+\At)^v_vx,\At^v_xv)
\,.
\end{multline*}
To conclude, observe that by Lemma \ref{Lemma4}
\begin{equation*}
(\nabla_v \At)^v_w = -\At^v_{(\T+\At)^v_vw}
\end{equation*}
and
\begin{equation*}
(\nabla_v \T)^v_x = -\T^v_{(\T+\At)^v_vx}
\,.
\end{equation*}
\end{proof}
\begin{theorem}[Generalised Gauss equation and dual Gauss equation]\label{gaussanddual}
For each vertical admissible $ v $ and vertical vectors $ s $, $ u $ and $ w $ at $ p $,
\begin{multline}\tag*{\{0'\}}\label{0p}
(R_v(w,u)s)^\top_v
=
\hat{R}_v(w,u)s
 +
\T^v_w \T^v_u s
-
\T^v_u \T^v_w s
\\
+
(\hat\nabla_u\hat Q)^v_ws - (\hat\nabla_w\hat Q)^v_us
+
\hat Q^v_u\hat Q^v_ws - \hat Q^v_w\hat Q^v_us
\\
-
P_v(w,s,\hat Q^v_uv+\T^v_uv)^\top_v
+
P_v(u,s,\hat Q^v_wv+\T^v_wv)^\top_v
\\
-
2C^\sharp_v(\mathbf T^v_ws,\hat Q^v_uv+\mathbf T^v_uv)^\top_v
+
2C^\sharp_v(\mathbf T^v_us,\hat Q^v_wv+\mathbf T^v_wv)^\top_v
\,. 
\end{multline}
 For each horizontal $v$ and $g_v$-horizontal vectors $x$, $y$ and $z$,
\begin{multline}\tag*{\{4'\}}\label{4p}
(R_v(x,y)z)^\bot_v
=
\tilde{R}^\ast_v(\tilde{x},\tilde{y})\tilde{z}
 +
\At^v_x \At^v_y z
- 
\At^v_y \At^v_x z 
+
\At^v_z\At^v_y x
-
\At^v_z\At^v_x y
\\
+
2 \At^v_z C^\sharp(y,\At^v_xv)^\top_v
+
2 \At^v_y C^\sharp_v(z,\At^v_xv)^\top_v
-
2 \At^v_z C^\sharp_v(x,\At^v_yv)^\top_v
-
2 \At^v_x C^\sharp_v(z,\At^v_yv)^\top_v
\\
+
((\nabla_x \tilde{Q})^v_y z - (\nabla_y \tilde{Q})^v_x z)^\bot_v
+
\tilde{Q}^v_y \tilde{Q}^v_x z - \tilde{Q}^v_x \tilde{Q}^v_y z
+
((\dot\partial\tilde Q)^v_yz(\At^v_xv))^\bot_v
-
((\dot\partial\tilde Q)^v_xz(\At^v_yv))^\bot_v
\\
-
P_v(y,z,\At^v_xv)^\bot_v
+
P_v(x,z,\At^v_yv)^\bot_v
\,.
\end{multline}
\end{theorem}
\begin{proof}
From \ref{0} and \ref{4} together with Propositions \ref{RTopVProposition} and \ref{RBotVProposition}.
\end{proof}
\begin{corollary}\label{verflagpole}
For $v\in\ver$ and $w$ vertical,  assuming that the plane $\{v,w\}$ is non-degenerate,
\begin{multline*}
K_v(w)
=
\hat K_v(w)
- 
\frac{g_v(\T^v_ww,\T^v_vv)-g_v(\T^v_vw,\T^v_vw)}{L(v)g_v(w,w)-g_v(v,w)^2}
\\
 - 
\frac{g_v( P_v (w,w,\T^v_vv)-(\hat\nabla_v\hat Q)^v_ww,v)+C_v(w,\hat Q^v_vu,\T^v_vv)}{L(v)g_v(w,w)-g_v(v,w)^2}
\end{multline*}
where $\hat K$ denotes the flag curvature intrinsic to the submersion fibers.
\end{corollary}
\begin{proof}
The proof is completely analogous to  \cite[Corollary 6]{2022HuberJavaloyes}.
\end{proof}
\begin{corollary}\label{horflagpole}
For $v\in\hor$ and $x$ $g_v$-horizontal, denoting by $\tilde v$ and $\tilde x$ their projections, and $w$ vertical, assuming that the planes $\{v,w\}$ and $\{v,x\}$ are non-degenerate,
\begin{align*}
K_v(w)&= \frac{g_v((\nabla_v\T)^v_ww,v)+g_v(\At^v_vw,\At^v_vw)-g_v(\T^v_wv,\T^v_wv))}{L(v)g_v(w,w)-g_v(v,w)^2},\\
K_v(x)&=\tilde K_{\tilde v}(\tilde x)
 - \frac{3g_v(\At^v_x v,\At^v_x v)}{L(v)g_v(x,x)-g_v(v,x)^2}
\,.
\end{align*}
\end{corollary}
\begin{proof}
The expression for $K_v(w)$ follows from $\{2\}$ in Corollary \ref{fundeqgeneral} taking into account that
\[g_v(\At^v_{\At^v_vw}w,v)=-g_v(\At^v_{\At^v_vw}v,w)=g_v(\At^v_v{\At^v_vw},w)=-g_v(\At^v_vw,\At^v_vw)\]
by skew-symmetry and part $(iii)$ of Proposition \ref{TAProposition}, and $g_v(\nabla_w\At)^v_vw,w)=0$. To check this take a horizontal extension $V$ with $\nabla^v_vV=0$, and $W$ a vertical extension of $w$, then
\begin{equation}\label{depaso}
g_v(\nabla_w\At)^v_vw,w)=g_v(\nabla^v_w(\At^V_VW),v)-g_v(\At^v_{\nabla^v_wV}w,v)-g_v(\At^v_v(\nabla^v_wW),v).
\end{equation}
Using again skew-symmetry and part $(iii)$ of Proposition \ref{TAProposition}, 
\[g_v(\At_v(\nabla^v_wW),v)=-g_v(\At_vv,\nabla^v_wW)=0,\]
\[g_v(\nabla^v_w(\At^V_VW),v)=w(g_V(\At^V_VW,V))-g_v(\At^v_vw,\nabla^v_wV)=-g_v(\At^v_vw,\At^v_vw),\]
since $g_V(\At^V_VW,V)=-g_V(W,\At^V_VV)=0$ and $(\nabla^v_wV)^\bot_v=\At^v_vw$ by \eqref{Lieconversion}. Finally, 
\begin{align*}
g_v(\At^v_{\nabla^v_wV}w,v)&=-g_v(\At^v_{\nabla^v_wV}v,w)=g_v(\At^v_v{(\nabla^v_wV)},w)=-g_v(\At^v_vw,{\nabla^v_wV})\\
&=
-g_v(\At^v_vw,\At^v_vw).
\end{align*}
Applying the last three identities in \eqref{depaso}, we conclude that $g_v(\nabla_w\At)^v_vw,w)=0$.

From \ref{4p}, by properties of the Cartan tensor,  part $(iii)$ of Proposition \ref{TAProposition}, 
 the identity $\tilde{Q}^v_v=\tilde{Q}^vv=0$ for $v$ horizontal,  and taking into account that $g_v(\At^v_v\At^v_xx,v)$ and  $g_v(\At^v_v C^\sharp_v(x,\At^v_xv)^\top_v,v)$ also vanish by skew-symmetry, we get 
\begin{align}\nonumber
g_v(R_v(v,x)x,v)
& =
g_v(\tilde{R}^\ast_v(\tilde{v},\tilde{x})\tilde{x}
+ 3 \At^v_x \At^v_x v
+ (\nabla_v \tilde{Q})^v_x x - (\nabla_x \tilde{Q})^v_v x,v)
\\ & \qquad
- g_v((\dot\partial\tilde Q)^v_vx(\At^v_xv)+ P_v(v,x,\At^v_xv),v)
\,.\label{eqrefer}
\end{align}
By definition of $\nabla\tilde Q$, for a locally horizontal admissible extension $V$ of $v$ and locally $g_V$-horizontal extension $X$ of $x$ satisfying $\nabla^v_xV=\At^v_xv$ and $\nabla^v_vV=0$
\begin{equation*}
(\nabla_v\tilde Q)^v_x x
=
\nabla^v_v\tilde Q^V_X X
-
\tilde Q^v_{\nabla^v_vX}x
-
\tilde Q^v_x\nabla^v_vX
\end{equation*}
such that, using that $g_v(\tilde Q^v_yz,v)=0$ for $v$ horizontal and arbitrary $y$ and $z$  by definition of $\tilde Q$ and properties of the Cartan tensor, and, in particular $g_V(\tilde Q^V_XX,V)=0$,  
\begin{equation*}
g_v((\nabla_v\tilde Q)^v_x x,v)
=
v(g_V(Q^V_XX,V))
-
g_v(\tilde Q^v_xx,\nabla^v_vV)
-
2C_v(\tilde Q^v_xx,v,\nabla^v_vV)
=
0
\,,
\end{equation*}
furthermore, since $Q^V_VX=0$,
\begin{multline*}
(\nabla_x\tilde Q)^v_vx
=
\nabla^v_x\tilde Q^V_V X
-
\tilde Q^v_{\nabla^v_xV}x
-
\tilde Q^v_v\nabla^v_xX
-
(\dot\partial\tilde Q)^v_vx(\At^v_xv)
=
-
\tilde Q^v_{\At^v_xv}x
-
(\dot\partial\tilde Q)^v_vx(\At^v_xv)
\\
=
-
(\dot\partial\tilde Q)^v_vx(\At^v_xv).
\end{multline*}
Finally, from \cite[Eq. (56)]{Jav20},
\begin{equation*}
g_v(P_v(v,x,\At^v_xv),v) = 0.
\end{equation*}
 All the above identities imply that in \eqref{eqrefer} all the terms to the right are zero but the first two. Taking into account that by skew-symmetry $g_v(\At^v_x \At^v_x v,v)=-g_v(\At^v_x v,\At^v_x v)$, we conclude. 
\end{proof}
\begin{remark}
 The expressions for the flag curvatures in the last corollary were previously obtained in \cite[Theorem 5.1]{Vitorio2010} (see also \cite[Theorem 5.12]{Vitorio2017} for the second one). In the first identity, the covariant derivative of $\T$ in this reference is slightly different from ours. 
On the other hand, in the positive definite case, $g_v(\At^v_xv,\At^v_xv)$ is non-negative. As a consequence of this corollary, we have just shown that Finsler submersions never   increase   the flag curvature along horizontal flags, in the sense that
\begin{equation}
K_v(x) \leq \tilde K_{\tilde v}(\tilde x)
\end{equation}
along each horizontal $v$ and for every $g_v$-horizontal $x$. This fact has been previously derived by a different method in \cite[Theorem 6.1]{2001AlDur}.
\end{remark}

\section{Submersions whose fibers are totally geodesic}\label{totallygeo}

 Let us study pseudo-Finsler submersions having fibers which are totally geodesic, namely, the geodesics of their fibers are also geodesics of the total space. In the classical case of Riemannian submersions, this corresponds with the submersions having $\T=0$, and they can be characterized by the fact that all the fibers are isometric. We will see that the Finslerian case is more complex.   Observe that the fibers of a pseudo-Finsler submersion are totally geodesic if and only if $\T^v_vv=0$ for all admissible vertical $v$. This turns out to be equivalent to $\T^v_vu=0$ for all admissible vertical $v$ and arbitrary $u$ (see \cite[Prop. 4]{2022HuberJavaloyes} and the basic properties of $\T$). 
\begin{proposition}
 Given two pseudo-Finsler submersions  $\sigma_1:(\M_1,L_1)\rightarrow (\M_2,L_2)$ and $\sigma_2:(\M_2,L_2)\rightarrow (\M_3,L_3)$, it holds that
\begin{enumerate}[(i)]
\item their composition is also a pseudo-Finsler submersion,
\item if this composition has totally geodesic fibers, then so does the submersion onto the final base.
\end{enumerate}
\end{proposition}
\begin{proof}
%
 To check $(i)$, let $g^1$ and $g^2$ be the fundamental tensors of $L_1$ and $L_2$, respectively. By hypothesis every $v$ horizontal with respect to the initial submersion satisfies $L_1(v)=L_2(\dif\sigma_1(v))$. We have a similar identity for the final submersion. By definition, a vector is horizontal if and only if $g_v(v,\cdot)$ vanishes on the vertical subspace; because the vertical subspace of the composition contains the vertical subspace of the initial submersion, a horizontal vector with respect to the composition is in particular a horizontal vector for the initial submersion, such that $L_1(v)=L_2(\dif\sigma_1(v))$ is satisfied for that initial submersion, whereby \eqref{gv1} yields that $g^1_v(v,e)=0$ if and only if $ g^2_{\dif\sigma_1(v)}(\dif\sigma_1(v),\dif\sigma_1(e))=0$ for each vector $e\in T_{\pi(v)}\M_1$. Now observe that $\dif\sigma_1(v)$ is horizontal with respect to the final submersion, which concludes, since then $L_1(v)=L_2(\dif\sigma_1(v))=L_3(\dif\sigma_2(\dif\sigma_1(v)))$. This is because if $\tilde e$ is $\sigma_2$-vertical, then as $\dif\sigma_1$ is surjective, $\tilde e=\dif\sigma_1(e)$ for some $e\in T\M_1$. Moreover, $e$ is $\sigma_1\circ\sigma_2$-vertical and as $v$ is  $\sigma_1\circ\sigma_2$-horizontal, it follows that 
$g^1_v(v,e)=0$. Applying \eqref{gv2}, we conclude that $g^2_{\dif\sigma_1(v)}(\dif\sigma_1(v),\tilde e)=0$, and then $\dif\sigma_1(v)$ is $\sigma_2$-horizontal as required.

 To check $(ii)$, let $\gamma_2$ be a geodesic of a fiber of $\sigma_2$ and $\gamma_1$ one of its horizontal lifts by 
$\sigma_1$. If $\nabla$ and $\tilde\nabla$  are the Chern connections of $(\M_1,L_1)$ and $(\M_2,L_2)$, respectively, and $D_{\gamma_1}$ and $\tilde D_{\gamma_2}$ their associated covariant derivatives alogn $\gamma_1 $ and $\gamma_2$, respectively, then
  \begin{equation}\label{dsigma1}
  \dif\sigma_1(D^{\dot\gamma_1}_{\gamma_1}\dot\gamma_1)=\tilde D^{\dot\gamma_2}_{\gamma_2}\dot\gamma_2
\end{equation}  
  by the dual Gauss formula (see Proposition \ref{GaussProposition} and observe that $\tilde Q^v_vv=0$ for all $v$ horizontal). Moreover, it is easy to check that $\gamma_1$ lies in a fiber of $\sigma_2\circ\sigma_1$, and  $D^{\dot\gamma_1}_{\gamma_1}\dot\gamma_1$ is $\sigma_2\circ\sigma_1$-vertical because  $\sigma_2\circ\sigma_1$ has totally geodesic fibers. Summing up, from the last claim and \eqref{dsigma1}, 
\[\dif\sigma_2(\tilde D^{\dot\gamma_2}_{\gamma_2}\dot\gamma_2)=\dif\sigma_1( \dif\sigma_1(D^{\dot\gamma_1}_{\gamma_1}\dot\gamma_1))=0,\]
and therefore $\tilde D^{\dot\gamma_2}_{\gamma_2}\dot\gamma_2$ is $\sigma_2$-vertical, which implies that $\T^{\dot\gamma_2}_{\dot\gamma_2}\dot\gamma_2=0$, with $\T$ the corresponding O'Neill tensor of $\sigma_2$, and therefore $\sigma_2$ has totally geodesic fibers.

\end{proof}
 Let  $\sigma:(\M,L)\rightarrow (\B,\tilde L)$  be a pseudo-Finsler submersion.  Assume that the domains of $L$ and $\tilde L$ are the slit tangent bundle and $\dim\M\geq 3$.  For each piecewise smooth path $ \tilde\gamma $ of the base manifold joining $\tilde p$ and $\tilde q$, let $ F_{\tilde\gamma} :\fiber_{\tilde p}\rightarrow \fiber_{\tilde q}$ be the map defined for every $p\in \fiber_{\sigma(p)}$ as follows. Let $\gamma$ be the horizontal lift of $\tilde \gamma$ with initial point at $p$, which there exists and it is unique by Lemma \ref{legendrelemma}. Then $F_{\tilde\gamma}(p)$ is defined as the endpoint of $\gamma$. As horizontal lifts are determined by ODE's, the map $F_{\tilde\gamma}$ will be smooth because of the smooth dependence of the solutions of these EDO's from initial values. Moreover, it is a diffeomorphism because if $\tilde\beta$ is the reverse curve of $\tilde \gamma$, then $F_{\tilde \beta}$ is the smooth inverse of $F_{\tilde\gamma}$. 
Due to the relationship between (ambient) horizontal geodesics and geodesics of the base manifold, we may prove that, under the right conditions, the map $F_{\tilde\gamma}$ is an isometry.
\begin{definition}\label{horreg}
A pseudo-Finsler submersion is said to be  horizontally regular when for each admissible vertical $ v $ and projectable horizontal $ X $
\begin{equation*}
(\nabla^v_v (X^\top))^\top_v = 0.
\end{equation*}
\end{definition}
\begin{proposition}\label{isometry}
  Let $\sigma:(\M,F)\rightarrow (\B,\tilde F)$ be a  horizontally  regular Finsler submersion between complete Finsler manifolds. Then  the map $ F_{\tilde\gamma} $  is an isometry for all $\tilde \gamma$ if and only if the fibers are totally geodesic.
\end{proposition}
\begin{proof}
 Assume first that the fibers are totally geodesic.   To show that $ F_{\tilde\gamma} $ is an isometry, assume that $ \tilde\gamma $ is parametrised in $[0,1]$ such that $ \tilde\gamma(0) = \tilde{p} $ and $ \tilde \gamma(1) = \tilde{q} $.  Consider a smooth  admissible  curve
$\alpha_0:[0,1]\rightarrow \fiber_{\tilde p}$ and then define $\lambda:[0,1]\times [0,1]\rightarrow \M$ such for every $t\in[0,1]$, $s\rightarrow \lambda(s,t)$ is the horizontal lift of $\tilde \gamma$ with initial point at $\alpha_{ 0}(t)$. 
 Denote  $\alpha_s$ the curve in the fiber $\fiber_{\tilde\gamma(s)}$ defined as $\alpha_s(t)=\lambda(s,t)$ and $\beta_t$ the horizontal  curve given by $\beta_t(s)=\lambda(s,t)$. Let $X_s(t)=\frac{d}{d s}\lambda(s,t)=\dot\beta_t(s)$ be one of the variational vector fields of $\lambda$. Observe that $X_s$ is always horizontal.  Then, recalling the commutation of the covariant derivative in a two-parameter map (see  \cite[Prop. 3.2]{Jav14}), 
\begin{equation}\label{previous}
(D^{\dot\alpha_s}_{\beta_t} \dot\alpha_s)^\top_{\dot\alpha_s}
=
(D^{\dot\alpha_s}_{\alpha_s} X_s)^\top_{\dot\alpha_s}
=
(D^{\dot\alpha_s}_{\alpha_s} X_s^\top)^\top_{\dot\alpha_s}
+
\T^{\dot\alpha_s}_{\dot\alpha_s} (X_s)^\bot_{\dot\alpha_s}
=
\T^{\dot\alpha_s}_{\dot\alpha_s} {(X_s)}^\bot_{\dot\alpha_s}
=
0
\end{equation}
by  horizontal  regularity and total geodesicity of the fibers, therefore 
\begin{equation*}
g_{\dot\alpha_s}(D^{\dot\alpha_s}_{\beta_t} \dot\alpha_s,\dot\alpha_s)
=
0
\,.
\end{equation*}
Let us now show that $ \alpha_1 = F_{\tilde\gamma} \circ \alpha_0 $ has the same length as $ \alpha_0 $. Assuming that $  F(\dot\alpha_s)^2 = g_{\dot\alpha_s}(\dot\alpha_s,\dot\alpha_s) $ is non-negative, and by the Leibniz integral rule, the variation
\begin{equation*}
\tfrac{\partial}{\partial s} \int^1_0 \sqrt{g_{\dot\alpha_s}(\dot\alpha_s,\dot\alpha_s)} \mathrm{d} t
=
\int^1_0 \tfrac{\partial}{\partial s} \sqrt{g_{\dot\alpha_s}(\dot\alpha_s,\dot\alpha_s)} \mathrm{d} t
=
\int^1_0 \tfrac{\frac{\partial}{\partial s} g_{\dot\alpha_s}(\dot\alpha_s,\dot\alpha_s)}{\sqrt{g_{\dot\alpha_s}(\dot\alpha_s,\dot\alpha_s)}} \mathrm{d} t
\end{equation*}
of the length of $ \alpha_s $ is zero. Indeed,
\begin{equation*}
\tfrac{\partial}{\partial s} g_{\dot\alpha_s}(\dot\alpha_s,\dot\alpha_s)
=
2 g_{\dot\alpha_s}(D^{\dot\alpha_s}_{\beta_t} \dot\alpha_s,\dot\alpha_s)
+
2 C_{\dot\alpha_s}(\dot\alpha_s,\dot\alpha_s,D^{\dot\alpha_s}_{\beta_t} \dot\alpha_s)
=
0.
\end{equation*}
Consequently, the image by $ F_{\tilde\gamma} $ of any vertical geodesic between two fixed points $ p $ and $ p' $ has the same length. A geodesic that attains the minimal length exists by completeness, 
therefore, given $p_1,p_2\in \fiber_{\tilde p} $, it follows that $d_{\fiber_{\tilde p}}(p_1,p_2)\geq d_{\fiber_{\tilde q}}(F_{\tilde\gamma} (p_1),F_{\tilde\gamma} (p_2))$ and reasoning with the inverse $F_{\tilde \beta}$, one obtains the other inequality, concluding that $F_{\tilde\gamma} $ is an isometry.

Conversely,   assume that the maps $F_{\tilde\gamma}$ are always isometries. Since by  the first two identities in \eqref{previous}  and skew-symmetry
\begin{equation*}
g_{\dot\alpha_s}(D^{\dot\alpha_s}_{\beta_t} \dot\alpha_s,\dot\alpha_s)
=
g_{\dot\alpha_s}(\T^{\dot\alpha_s}_{\dot\alpha_s} X_s,\dot\alpha_s)
=
-
g_{\dot\alpha_s}(X_s,\T^{\dot\alpha_s}_{\dot\alpha_s} \dot\alpha_s)
\,,
\end{equation*}
we have that if
\begin{equation*}
\int^1_0 \tfrac{\frac{\partial}{\partial s} g_{\dot\alpha_s}(\dot\alpha_s,\dot\alpha_s)}{\sqrt{g_{\dot\alpha_s}(\dot\alpha_s,\dot\alpha_s)}} \mathrm{d} t
=
\tfrac{\partial}{\partial s} \int^1_0 \sqrt{g_{\dot\alpha_s}(\dot\alpha_s,\dot\alpha_s)} \mathrm{d} t
=
0
\end{equation*}
then, by the mean value theorem, for some $t_1\in[0,1]$
\begin{equation*}
g_{\dot\alpha_s(t_1)}(X_s(t_1),\T^{\dot\alpha_s(t_1)}_{\dot\alpha_s(t_1)} \dot\alpha_s(t_1))
=
\tfrac{\partial}{\partial s} g_{\dot\alpha_s(t_1)}(\dot\alpha_s(t_1),\dot\alpha_s(t_1))
=
0
\,.
\end{equation*}
By the same argument we obtain a sequence of values $t_n\in[0,\frac{1}{2^n}]$ converging to $0$ for which $g_{\dot\alpha_s(t_n)}(X_s(t_n),\T^{\dot\alpha_s(t_n)}_{\dot\alpha_s(t_n)} \dot\alpha_s(t_n))$ is zero, and by continuity
\begin{equation*}
g_{\dot\alpha_s(0)}(X_s(0),\T^{\dot\alpha_s(0)}_{\dot\alpha_s(0)} \dot\alpha_s(0))
=
0
\,.
\end{equation*}
Conclude by non-degeneracy.
\end{proof}
\begin{corollary}
For a horizontally regular Finsler submersion with totally geodesic fibers, the flow of a projectable horizontal geodesic vector field gives rise to an isometry between the fibers.
\end{corollary}
\begin{proof}
Consider on some neighbourhood a projectable vector field satisfying the geodesic equation. Provided it is horizontal, each of the geodesic arcs described by its flow is the horizontal lift of some geodesic arc $ \tilde\gamma $ described by the flow of its projection, which satisfies the geodesic equation by the dual Gauss formula; by construction $ F_{\tilde\gamma} $ is that horizontal flow, which we have just shown to be an isometry between the fibers  in Proposition \ref{isometry}. 
\end{proof}
\begin{theorem}\label{totgeoth}
If the ambient manifold of a Finsler submersion is connected and geodesically complete, then its base manifold is geodesically complete. If additionally the submersion is horizontally regular and has totally geodesic fibers, then it is the projection of a bundle associated with a principal fiber bundle whose structure group is the Lie group of isometries of the fiber.
\end{theorem}
\begin{proof}
First recall that the isometries of a Finsler manifold form a Lie group (see \cite{2002DengHou} for further details).

For completeness, observe that each geodesic arc is extended by the projection of the extension of its horizontal lift. For the second part of the theorem, fix a point $ \tilde{p} $ in the base manifold $ \B $ and denote by $ \mathbf{F} $ its submersion fiber  at some point $\tilde p$, namely, $\mathbf{F}=\fiber_{\tilde p}$.  Let us futhermore denote by $ \mathbf{G} $ the Lie group of isometries of $ \mathbf{F} $, by $ \mathbf{G}_{\tilde{q}} $ the set of isometries from $ \mathbf{F} $ to $ \sigma^{-1}(\lbrace \tilde{q} \rbrace) $, and by $ \mathbf{E} $ their union. Note how $ \mathbf{G} $ acts (diffeomorphically)  and freely  on $ \mathbf{E} $ as
\begin{equation*}
\begin{array}{ccc}
\mathbf{G} \times \mathbf{E} & \to & \mathbf{E}
\\
(\mathbf{g},\mathbf{e}) & \mapsto & \mathbf{e} \circ \mathbf{g}
\mathrlap{\,.}
\end{array}
\end{equation*}
 Let us show that $ \mathbf{E} $ is a fiber bundle over $ \B $ for the map $ \mu \colon \mathbf{E} \to \B $ sending each $ \mathbf{G}_{\tilde{q}} $ onto $ \tilde{q} $, with structure group $ \mathbf{G} $, equipped with the differentiable structure obtained as follows.

For the choice of a point $ \tilde{q}_i $ in each set of some open covering of $ \B $ by geodesically convex sets $ \mathbf{U}_i $, and the choice of a smooth path $ \tilde\gamma_i $ joining $ \tilde{p} $ to $ \tilde{q}_i $, there exists, at each point $ \tilde{q} $ and for each $ \mathbf{U}_i $ that contains it, a unique geodesic arc $ \tilde\gamma_{i,\tilde{q}} $ from $ \tilde{q}_i $ to $ \tilde{q} $. By the previous proposition, both $ \tilde\gamma_i $ and $ \tilde\gamma_{i,\tilde{q}} $ induce an isometry between the corresponding submersion fibers, respectively $ F_{\tilde\gamma_i} $ and $ F_{\tilde\gamma_{i,\tilde{q}}} $. Let us define a local section $ \mathbf{e}_i $ on each $ \mathbf{U}_i \ni \tilde{q} $ as the unique isometry
\begin{equation*}
\mathbf{e}_i(\tilde{q}) = F_{\tilde\gamma_{i,\tilde{q}}} \circ F_{\tilde\gamma_i} \in \mathbf{G}_{\tilde{q}}
\,.
\end{equation*}
When $ \mathbf{U}_i \cap \mathbf{U}_j \ni \tilde{q} $, we can define the transition functions
\begin{equation*}
\mathbf{g}_{i\!j}(\tilde{q}) = \mathbf{e}^{-1}_j(\tilde{q}) \circ \mathbf{e}_i(\tilde{q}) \in \mathbf{G}
\,,
\end{equation*}
satisfying
\begin{equation*}
\mathbf{e}_i(\tilde{q}) = \mathbf{e}_j(\tilde{q}) \circ \mathbf{g}_{i\!j}(\tilde{q})
\end{equation*}
and the cocycle condition with respect to the group action
\begin{equation*}
\mathbf{g}_{ik}(\tilde{q}) = \mathbf{g}_{jk}(\tilde{q}) \circ \mathbf{g}_{i\!j}(\tilde{q})
\,.
\end{equation*}
By the fiber bundle construction theorem, $ \mathbf{E} $ is a fiber bundle over $ \mathcal{B} $ for the map $ \mu $.
Denote by $ \mathbf{E}' $ the quotient of $ \mathbf{E} \times \mathbf{F} $ by the (diffeomorphic) action
\begin{equation*}
\begin{array}{ccc}
\mathbf{G} \times \mathbf{E} \times \mathbf{F} & \to & \mathbf{E} \times \mathbf{F}
\\ 
(\mathbf{g},(\mathbf{e},p)) & \mapsto & (\mathbf{e} \circ \mathbf{g},\mathbf{g}^{-1}(p))
\mathrlap{\,.}
\end{array}
\end{equation*}
 By the definition of the action, the map
\begin{equation*}
\begin{array}{ccc}
\mathbf{E} \times \mathbf{F} & \to & \B
\\
(\mathbf{e},p) & \mapsto & \mu(\mathbf{e})
\end{array}
\end{equation*}
is identically equal to $\tilde{q}$ on the whole equivalence class of $(\mathbf{e},p)$ in $ \mathbf{E}' $, and induces a well-defined projection $ \mu' \colon \mathbf{E}' \to \B $, defining a fiber bundle associated with the principal fiber bundle $ \mu \colon \mathbf{E} \to \B $. Let us now show that the submersion coincides with this associated fiber bundle.
Indeed, the action on $ \mathbf{E} \times \mathbf{F} $ is free, by freedom of the action of $ \mathbf{G} $ on $ \mathbf{E} $  and proper, because the action of $G$ on $\mathbf F$ is proper (recall that the isometries of a Finsler manifold is a closed subgroup of the isometries of the Riemannian metric obtained by averaging \cite{Torrome}). 
The quotient $ \mathbf{E}' $ of $ \mathbf{E} \times \mathbf{F} $ by the free and proper action of $ \mathbf{G} $ is a manifold by the quotient manifold theorem (see  \cite[Th. 7.10]{2000Lee}).  Finally, the map
\[ \phi: \mathbf{E}\times \mathbf{F}\rightarrow\M\]
defined as $\phi(\mathbf{e},p)=\mathbf{e}(p)$ induces a map $\phi'$ from the quotient $\mathbf{E}'$, which is in fact a diffeomorphism and maps fibers of $\mu'$ to fibers of $\sigma$. To check that $\phi'$ is surjective observe that fixing $\mathbf{e}\in\mathbf{E}$, $\phi'(\mathbf{e},\cdot)$ maps $\mathbf{F}$ to a fiber of $\M$ diffeomorphically. Moreover, if 
$\phi(\mathbf{e},p)=\phi(\mathbf{e}',p')$, then $\mathbf{e}^{-1}\circ\mathbf{e}'\in \mathbf{G}$ and 
$(\mathbf{e},p)$ and $(\mathbf{e}', p')$ are representatives of the same class under the action of $\mathbf{G}$. 
\end{proof}


\begin{thebibliography}{10}

\bibitem{AAD19}
{\sc M.~M. Alexandrino, B.~O. Alves, and H.~R. Dehkordi}, {\em On {F}insler
  transnormal functions}, Differential Geom. Appl., 65 (2019), pp.~93--107.

\bibitem{AAJ19}
{\sc M.~M. Alexandrino, B.~O. Alves, and M.~A. Javaloyes}, {\em On singular
  {F}insler foliation}, Ann. Mat. Pura Appl. (4), 198 (2019), pp.~205--226.

\bibitem{AEi22}
{\sc M.~M. Alexandrino, F.~M. Escobosa, and M.~K. Inagaki}, {\em Traveling
  along horizontal broken geodesics of a homogenous {F}insler submersion},
  arXiv:2204.13218 [math.DG],  (2022).

\bibitem{2001AlDur}
{\sc J.~C. \'{A}lvarez Paiva and C.~E. Dur\'{a}n}, {\em Isometric submersions
  of {F}insler manifolds}, Proc. Amer. Math. Soc., 129 (2001), pp.~2409--2417.

\bibitem{Besse}
{\sc A.~L. Besse}, {\em Einstein manifolds}, vol.~10 of Ergebnisse der
  Mathematik und ihrer Grenzgebiete (3) [Results in Mathematics and Related
  Areas (3)], Springer-Verlag, Berlin, 1987.

\bibitem{2002DengHou}
{\sc S.~Deng and Z.~Hou}, {\em The group of isometries of a {F}insler space},
  Pacific J. Math., 207 (2002), pp.~149--155.

\bibitem{Vitorio2017}
{\sc C.~Durán and H.~Vit\'orio}, {\em Moving planes, {J}acobi curves and the
  dynamical approach to {F}insler geometry}, European Journal of Mathematics, 3
  (2017).

\bibitem{FaIaPa04}
{\sc M.~Falcitelli, S.~Ianus, and A.~M. Pastore}, {\em Riemannian submersions
  and related topics}, World Scientific Publishing Co., Inc., River Edge, NJ,
  2004.

\bibitem{1967Gray}
{\sc A.~Gray}, {\em Pseudo-{R}iemannian almost product manifolds and
  submersions}, J. Math. Mech., 16 (1967), pp.~715--737.

\bibitem{HCSR21}
{\sc Q.~He, Y.~Chen, S.~Yin, and T.~Ren}, {\em Isoparametric hypersurfaces in
  {F}insler space forms}, Sci. China Math., 64 (2021), pp.~1463--1478.

\bibitem{HPV19}
{\sc M.~Hohmann, C.~Pfeifer, and N.~Voicu}, {\em {F}insler gravity action from
  variational completion}, Phys. Rev. D,  (2019), p.~064035.

\bibitem{2022HuberJavaloyes}
{\sc M.~Huber and M.~A. Javaloyes}, {\em The flag curvature of a submanifold of
  a {R}anders-{M}inkowski space in terms of {Z}ermelo data}, Results Math., 77
  (2022), pp.~Paper No. 124, 33.

\bibitem{Jav14}
{\sc M.~A. Javaloyes}, {\em Chern connection of a pseudo-{F}insler metric as a
  family of affine connections}, Publ. Math. Debrecen, 84 (2014), pp.~29--43.

\bibitem{Jav19}
{\sc M.~A. Javaloyes},  {\em Anisotropic tensor
  calculus}, Int. J. Geom. Methods Mod. Phys., 16 (2019), pp.~1941001, 26.

\bibitem{Jav20}
{\sc M.~A. Javaloyes}, {\em Curvature computations in {F}insler geometry using
  a distinguished class of anisotropic connections}, Mediterr. J. Math., 17
  (2020), pp.~Paper No. 123, 21.

\bibitem{JavSan14}
{\sc M.~A. Javaloyes and M.~S\'{a}nchez}, {\em On the definition and examples
  of {F}insler metrics}, Ann. Sc. Norm. Super. Pisa Cl. Sci. (5), 13 (2014),
  pp.~813--858.

\bibitem{JSV22}
{\sc M.~A. Javaloyes, M.~S\'anchez, and F.~F. Villase\~nor}, {\em Anisotropic
  connections and parallel transport in {F}insler spacetimes}, arXiv:2107.05986
  [math.DG],  (2021).

\bibitem{JSV21}
{\sc M.~A. Javaloyes, M.~S\'anchez, and F.~F. Villaseñor}, {\em The
  {E}instein-{H}ilbert-{P}alatini formalism in pseudo-{F}insler geometry},
  arXiv:2108.03197 [math.DG], to appear in Advances in Theoretical and
  Mathematical Physics,  (2021).

\bibitem{2000Lee}
{\sc J.~M. Lee}, {\em Introduction to smooth manifolds}, vol.~218 of Graduate
  Texts in Mathematics, Springer, New York, second~ed., 2013.

\bibitem{Min15}
{\sc E.~Minguzzi}, {\em Light cones in {F}insler spacetime}, Comm. Math. Phys.,
  334 (2015), pp.~1529--1551.

\bibitem{1966ONeill}
{\sc B.~O'Neill}, {\em The fundamental equations of a submersion}, Michigan
  Math. J., 13 (1966), pp.~459--469.

\bibitem{1983ONeill}
{\sc B.~O'Neill}, {\em Semi-{R}iemannian
  geometry}, vol.~103 of Pure and Applied Mathematics, Academic Press, Inc.
  [Harcourt Brace Jovanovich, Publishers], New York, 1983.
\newblock With applications to relativity.

\bibitem{RuSu15}
{\sc M.~Ruzhansky and M.~Sugimoto}, {\em On global inversion of homogeneous
  maps}, Bull. Math. Sci., 5 (2015), pp.~13--18.

\bibitem{Torrome}
{\sc R.~G. Torromé}, {\em Average structures associated to a {F}insler space},
  arXiv:math/0501058 [math.DG],  (2005).

\bibitem{Vitorio2010}
{\sc H.~Vit\'orio}, {\em A Geometria de Curvas Fanning e de suas Redu\c c\~oes
  Simpl\'eticas}, PhD. Thesis, Universidade Estadual de Campinas, 2010.

\bibitem{Xu18}
{\sc M.~Xu}, {\em Isoparametric hypersurfaces in a {R}anders sphere of constant
  flag curvature}, Ann. Mat. Pura Appl. (4), 197 (2018), pp.~703--720.

\bibitem{XuZh18}
{\sc M.~Xu and L.~Zhang}, {\em $\delta$-homogeneity in {F}insler geometry and
  the positive curvature problem}, Osaka J. Math., 55 (2018), pp.~177--194.

\end{thebibliography}
\end{document}